\documentclass[12pt]{amsart}

\usepackage{amssymb}
\usepackage{verbatim}
\usepackage[toc,page]{appendix}
\usepackage{mathrsfs}
\usepackage{mathtools}

\mathtoolsset{showonlyrefs}

\newtheorem{thm}{Theorem}[section]

\newtheorem{lem}[thm]{Lemma}
\newtheorem{cor}[thm]{Corollary}

\theoremstyle{definition}

\theoremstyle{remark}

\newtheorem{remark}[thm]{Remark}

\numberwithin{equation}{section}

\newcommand{\R}{\mathbf{R}}
\newcommand{\N}{\mathbf{N}}
\newcommand{\C}{\mathbf{C}}
\newcommand{\Z}{\mathbf{Z}}

\newcommand{\Mod}[1]{\ (\textup{mod}\ #1)}
\providecommand{\SL}{\operatorname{SL}}

\begin{document}

\title[]{Spectral decomposition formula and moments of symmetric square $L$-functions}

\author{Olga  Balkanova}
\address{Steklov Mathematical Institute of Russian Academy of Sciences,  8 Gubkina st., Moscow, 119991, Russia}
\email{balkanova@mi-ras.ru}

\begin{abstract}
We prove a spectral decomposition formula for averages of Zagier $L$-series in terms of moments of symmetric square $L$-functions associated to Maass and holomorphic cusp forms of levels $4$, $16$, $64$. 
\end{abstract}

\keywords{L-functions; Gauss sums; Eisenstein series; Kuznetsov trace formula}
\subjclass[2010]{Primary: 11F12; 11F30; 11M99}

\maketitle



\section{Introduction}

The aim of this paper is to prove a spectral decomposition formula for the average 
\begin{equation}\label{average:l}
\sum_{l=1}^{\infty}\omega(l)\mathscr{L}_{n^2-4l^2}(s),
\end{equation}
where $\omega$ is a suitable test function and the $L$-series is defined as
\begin{equation}\label{Lbyk}
\mathscr{L}_{n}(s)=\frac{\zeta(2s)}{\zeta(s)}\sum_{q=1}^{\infty}\frac{b_q(n)}{q^{s}}
\end{equation}
for $ \Re{s}>1$ and can be meromorphically continued to the whole complex plane, see \cite[Proposition 3]{Z}.
Here $\zeta(s)$ denotes the Riemann zeta function and
\begin{equation*}
b_q(n):=\#\{x\Mod{2q}:x^2\equiv n\Mod{4q}\}.
\end{equation*}

The dual problem of investigating the average over $n$
\begin{equation}\label{average:n}
\sum_{n=1}^{\infty}\omega(n)\mathscr{L}_{n^2-4l^2}(s)
\end{equation}
 was studied in \cite{BF2} in connection with the prime geodesic theorem.  See also \cite{BFR}, \cite{BBCL}, \cite{Byk},  \cite{SY}, \cite{WZ} for related results.
  Furthermore, sums of the form \eqref{average:n} appear in the explicit formulas for the first moments of symmetric square $L$-functions associated to holomorphic (see \cite{Z}, \cite{BF1}) or Maass (see \cite{Bal}) cusp forms.
 These explicit formulas can serve as a starting point for analyzing second moments of symmetric square $L$-functions.
For example, if we take the first moment of Maass form symmetric square $L$-functions for $\SL_{2}(\Z)$ twisted by the Fourier coefficient $\rho_j(l^2)$
\begin{equation}
\sum_{j}\rho_j(l^2)L(s,\text{sym}^2 u_j),
\end{equation}
 multiply it by $\zeta(2s)l^{-s}$ and sum over $l$ from $1$ to $\infty$,
 we obtain the second moment 
 \begin{equation}
\sum_{j}L(s,\text{sym}^2 u_j)^2.
\end{equation}
Applying these manipulations to the explicit formula proved in \cite{Bal}, we discover  expressions of the following form
\begin{equation}\label{eq:doublesum}
\sum_{l=1}^{\infty}\sum_{n=1}^{\infty}\mathscr{L}_{n^2-4l^2}(s)f(n,l;s)
\end{equation}
 on the right-hand side of the explicit formula for the second moment.
 A possible approach to evaluating \eqref{eq:doublesum} consists in using spectral methods.
 However, spectral decomposition for the inner sum over $n$ results in a loop bringing us back to the moment we started from.
 To avoid this problem, we can change the order of summation and investigate the average over $l$ first.
 This is the main reason behind our interest in a spectral decomposition formula for \eqref{average:l}.

 Even though, the averages \eqref{average:l} and \eqref{average:n} look similar, there are some important differences in approach.
 Spectral decomposition of \eqref{average:n} is a relatively straightforward application of the Kuznetsov trace formula to the generalized Kloosterman sums, while 
in case of \eqref{average:l} our method relies  heavily on various  properties of Gauss sums.  
It is interesting to note that Gauss sums occur naturally in various papers on second moments of symmetric square $L$-functions, see \cite{Blo}, \cite{IM}, \cite{KhY}.
In our case, we express the average  \eqref{average:l} in terms of sums of a product of two Gauss sums. 
Evaluating these Gauss sums, we obtain sums of Kloosterman sums for $ \Gamma_0(N)$ with $N=4,16, 64$ at various cusps and twisted by $\chi_4$ (non-trivial Dirichlet character modulo $4$).
Finally, applying the Kuznetsov trace formula we derive a spectral decomposition formula  for  \eqref{average:l} which contains moments of symmetric square $L$-functions for $ \Gamma_0(N)$ with $N=4,16, 64$ twisted by Fourier coefficients at the cusps $0$ and $\infty$.

To state the main results rigorously, we introduce the function 
\begin{equation}
\psi(x)=\psi(x;n;s)=\frac{2}{\sqrt{\pi}}\left(\frac{x}{n}\right)^s\int_{0}^{\infty}\omega(y)\cos\left( \frac{2xy}{n}\right)dy
\end{equation}
and denote by $\psi_{H}(x)$ and $\psi_{D}(x)$ the Bessel integral transforms of $\psi(x)$ appearing in the Kuznetsov trace formula, see \eqref{eq:psiHk} and \eqref{eq:psiDt} . These transforms can be expressed in terms of the Gauss hypergeometric function as shown in Lemmas \ref{lem:psiH} and \ref{lem:psiD}.  
Let $\omega \in C^{\infty}$ be a function of compact support on $[a_1,a_2]$ for some $0<a_1<a_2<\infty$ and let $\widehat{\omega}$ stand for its Mellin transform.
It is also required to define the generalized divisor function
\begin{equation}\label{def:sigmas}
\sigma_s(\chi;n):=\sum_{d|n}\chi(d)d^{s}.
\end{equation}

For a cusp $\mathfrak{a}$ of $\Gamma_0(N)$, let us introduce the following notation
\begin{equation}\label{moment}
\mathfrak{M}_{\mathfrak{a}}(n,N,s)=\mathfrak{M}^{hol}_{\mathfrak{a}}(n,N,s)+\mathfrak{M}^{disc}_{\mathfrak{a}}(n,N,s),
\end{equation}
where
\begin{equation}
\mathfrak{M}^{hol}_{\mathfrak{a}}(n,N,s):=\sum_{\substack{k>1\\k \text{ odd}}}\psi_{H}(k)\Gamma(k)
\sum_{f\in H_k(N,\chi_4)}\rho_{f_{\mathfrak{a}}}\left(n\right)\overline{L(s,\text{sym}^2 f_{\infty})},
\end{equation}
\begin{equation}
\mathfrak{M}^{disc}_{\mathfrak{a}}(n,N,s):=\sum_{f\in H(N,\chi_4)}\frac{\psi_{D}(t_f)}{\cosh(\pi t_f)}\rho_{f_{\mathfrak{a}}}(n)\overline{L(s,\text{sym}^2 f_{\infty})}
\end{equation}
are the moments of symmetric square $L$-functions associated to holomorphic and Maass cusp forms of level $N$ with nebentypus $\chi_4$
twisted by the Fourier coefficient $\rho_{f_{\mathfrak{a}}}(n)$ of $f$ at a cusp $\mathfrak{a}$. The bar over $L$-functions means complex conjugation and $ f_{\infty}$ means that the L-functions are formed using the Fourier coefficients of $f$ around the cusp $\infty$.

\begin{thm}\label{thm:main1} 
Assume that $n$ is even.  For $0<\Re{s}<1$ the following explicit formula holds
\begin{multline}
\sum_{l=1}^{\infty}\omega(l)\mathscr{L}_{n^2-4l^2}(s)=M^{D}_{\text{even}}(n,s)+M^{C}(n,s)+\mathfrak{C}(n,s)\\
-\frac{2^{1-s}\pi^{1/2-s}i}{1-2^{-2s}}\mathfrak{M}_{\infty}(n^2/4,4,s)+\frac{2\pi^{1/2-s}}{1-2^{-2s}}\mathfrak{M}_{0}(n^2/4,4,s),
\end{multline}
where
\begin{equation}\label{eq:mevendef}
M^{D}_{\text{even}}(n,s)=\frac{\widehat{\omega}(1)\zeta(2s)}{L(\chi_4,1+s)}\left[n^{-2s}\sigma_s\left(\chi_4;n^2\right)+\sigma_{-s}(\chi_4;n^2)\right],
\end{equation}
\begin{multline}\label{eq:mcns} 
M^{C}(n,s)=
\frac{\Gamma(s-1/2)}{2^{s-1}\pi^{s-1/2}}\left(\sigma_{s-1}(\chi_4;n^2)+\frac{\sigma_{1-s}(\chi_4;n^2)}{n^{2-2s}} \right)\\ \times 
\frac{\zeta(2s-1)}{L(\chi_4,2-s)}
\biggl(
\sin (\pi s/2)\int_{0}^{n/2}\omega(y)\left( \frac{n^2}{4}-y^2\right)^{1/2-s}dy\\
+\cos(\pi s/2)\int_{n/2}^{\infty}\omega(y)\left(y^2- \frac{n^2}{4}\right)^{1/2-s}dy\biggr),
\end{multline}
\begin{multline}\label{eq:cns}
\mathfrak{C}(n,s)=\frac{L(\chi_4,s)}{4\pi^{s-1/2}}
\frac{1}{2\pi i}\int_{-\infty}^{\infty}\frac{\psi_D(t)\sinh(\pi t)}{t\cosh{(\pi t)}} \\
\times \left(n^{2it}\sigma_{-2it}(\chi_4;n^2)+n^{-2it}\sigma_{2it}(\chi_4;n^2)\right)
\frac{\zeta(s+2it)\zeta(s-2it)}{L(\chi_4,1+2it)L(\chi_4,1-2it)} dt.
\end{multline}

\end{thm}


\begin{thm}\label{thm:main2} 
Assume that $n$ is odd.  For $0<\Re{s}<1$ the following explicit formula holds
\begin{multline}
\sum_{l=1}^{\infty}\omega(l)\mathscr{L}_{n^2-4l^2}(s)=M^{D}_{\text{odd}}(n,s)+\frac{1}{2}M^{C}(n,s)+\frac{1}{2}\mathfrak{C}(n,s)\\
+\frac{8\pi^{1/2-s}}{1-2^{-2s}}\mathfrak{M}_{0}(n^2,64,s)+\frac{4\pi^{1/2-s}}{1+2^{-s}}\mathfrak{M}_{0}(n^2,16,s),
\end{multline}
where
\begin{equation}\label{eq:mdodddef}
M^{D}_{\text{odd}}(n,s)=\frac{\widehat{\omega}(1)\zeta(2s)}{L(\chi_4,1+s)}\sigma_{-s}(\chi_4;n^2),
\end{equation}
$M^{C}(n,s)$ is defined by \eqref{eq:mcns} and $\mathfrak{C}(n,s)$ by \eqref{eq:cns}.

\end{thm}

\begin{remark}
Note that  the main terms $$M^{D}_{\text{even}}(n,s)+M^{C}(n,s), \quad M^{D}_{\text{odd}}(n,s)+\frac{1}{2}M^{C}(n,s)$$  are holomorphic at the central point $s=1/2$. See Section \ref{section:holomorphic} for details.
\end{remark}

The paper is organized as follows. In Section \ref{sec:prelim} we collect all required tools and preliminary results. In Section \ref{sect:diagnondiag}, assuming that $\Re{s}$ is sufficiently large, we isolate the diagonal and non-diagonal terms for \eqref{average:l}, compute the diagonal term explicitly, and prove an expression for the non-diagonal term which is suitable for application of the Kuznetsov trace formula. Section \ref{sec:special} is devoted to the analysis of the Bessel integral transforms $\psi_{H}(x)$ and $\psi_{D}(x)$ appearing after the Kuznetsov trace formula is applied. More precisely, we show how to express $\psi_{H}(x)$ and $\psi_{D}(x)$ in terms of the Gauss hypergeometric functions. Sections \ref{sec:cont1} and \ref{sec:cont2} are concerned with evaluation of the continuous spectrum, while 
the holomorphic and discrete spectra are studied in Section \ref{sect:dischol}.
Finally, in Section \ref{sec:mainthm} we complete the proof of Theorems \ref{thm:main1}  and \ref{thm:main2} and compute the main terms at the central point.
\section{Preliminaries}\label{sec:prelim}

\subsection{Generalized divisor function}
In this subsection we collect various results related to the function $\sigma_s(\chi;n)$ defined by \eqref{def:sigmas}.

Note that the derivative of $\sigma_{s}(\chi;n)$ with respect to $s$ is equal to
\begin{equation}
\sigma'_{s}(\chi;n)=\sum_{d|n}\chi(d)d^s\log{d}.
\end{equation}

Let $\chi_4$ be a non-trivial Dirichlet character modulo $4$, so that
\begin{equation}
\chi_4(1)=1, \quad \chi_4(3)=-1,
\end{equation}
and
\begin{equation}\label{eq:sigmachi4}
\sigma_s\left(\chi_4; \left(\frac{n}{2} \right)^2\right)=\sigma_s(\chi_4;n^2).
\end{equation}

\begin{lem}
For odd $n$ the following identity holds
\begin{equation}\label{eq:sigmatwist}
\sigma_{1/2-u}(\chi_4;n^2)=n^{1-2u}\sigma_{-1/2+u}(\chi_4;n^2).
\end{equation}
\end{lem}
\begin{proof}
Let $n^2=bd$. Since $n$ is odd
\begin{equation}
1=\chi_4(n^2)=\chi_4(bd).
\end{equation}
Thus $\chi_4(b)=\chi_4(d)$.
Consequently,
\begin{multline}
\sigma_{1/2-u}(\chi_4;n^2)=\sum_{d|n^2}d^{1/2-u}\chi_4(d)=\sum_{bd=n^2}\left(\frac{n^2}{b} \right)^{1/2-u}\chi_4(d)\\=n^{1-2u}\sum_{bd=n^2}b^{-1/2+u}\chi_4(b)=n^{1-2u}\sigma_{-1/2+u}(\chi_4;n^2).
\end{multline}

\end{proof}

Consider the Dirichlet series:
\begin{equation}
Z(z,s):=\sum_{n=1}^{\infty}\frac{\sigma_s(\chi_4;n^2)}{n^z}.
\end{equation}
\begin{lem} We have
\begin{equation}\label{eq:zs}
Z(z,s)=\frac{1-2^{2s-z}}{1-2^{2s-2z}}\frac{L(\chi_4,z-s)\zeta(z)\zeta(z-2s)}{\zeta(2z-2s)}.
\end{equation}
\end{lem}
\begin{proof}
First, assume that $\Re{z}>1+2\Re{s}$. The Euler product for  $Z(z,s)$ is equal to
\begin{multline}\label{eq:zzs}
Z(z,s)=\prod_{p}\left( 1+\frac{\sigma_s(\chi_4;p^2)}{p^z}+\frac{\sigma_s(\chi_4;p^4)}{p^{2z}}+\ldots\right)\\
=\left( 1+\sum_{k=1}^{\infty}\frac{\sigma_s(\chi_4;2^{2k})}{2^{kz}}\right)\prod_{p>2}\left(1+\sum_{k=1}^{\infty}\frac{\sigma_s(\chi_4;p^{2k})}{p^{kz}} \right).
\end{multline}
Note that $\chi_4(d)=0$ for even $d$, and therefore, 
\begin{equation*}
\sigma_{s}(\chi_4;2^{2k})=\sum_{d|2^{2k}}d^s\chi_4(d)=1.
\end{equation*}
Consequently,
\begin{equation}\label{eq:firstmult}
 1+\sum_{k=1}^{\infty}\frac{\sigma_s(\chi_4;2^{2k})}{2^{kz}}=\frac{1}{1-2^{-z}}.
\end{equation}
Next, we evaluate the second multiple on the right hand side of \eqref{eq:zzs}. Since $p^2\equiv 1\pmod{4}$ we have
\begin{equation}
\chi_4(p^{2m})=1, \quad \chi_4(p^{2m+1})=\chi_4(p) \text{ for any }m \in \N.
\end{equation}
Therefore,
\begin{equation}
\sigma_s(\chi_4;p^{2k})=\frac{(p^{2s})^{k+1}-1}{p^{2s}-1}+\frac{(p^{2s})^k-1}{p^{2s}-1}\chi_4(p)p^{s}.
\end{equation}
This implies that
\begin{equation}
\sum_{k=1}^{\infty}\frac{\sigma_s(\chi_4;p^{2k})}{p^{kz}}=\frac{p^{2s}+\chi_4(p)p^{s}}{(p^{2s}-1)(p^{z-2s}-1)}-\frac{1+\chi_4(p)p^s}{(p^{2s}-1)(p^z-1)}.
\end{equation}
Using the property $\chi_{4}^2(p)=1$ we infer
\begin{multline}\label{eq:secondmult}
1+\sum_{k=1}^{\infty}\frac{\sigma_s(\chi_4;p^{2k})}{p^{kz}}=\frac{1+\chi_4(p)/p^{z-s}}{(1-1/p^{z-2s})(1-1/p^z)}\\
=\frac{1-1/p^{2(z-s)}}{(1-1/p^{z-2s})(1-1/p^z)(1-\chi_4(p)/p^{z-s})}.
\end{multline}
Finally, substituting  \eqref{eq:firstmult} and \eqref{eq:secondmult} in \eqref{eq:zzs} we prove the lemma.
\end{proof}




\subsection{Cusps and Kloosterman sums}

For a positive integer $N$, let $\Gamma=\Gamma_0(N)$ denote the Hecke congruence subgroup of level $N$.

The stabilizer of the cusp $\mathfrak{a}$ in $\Gamma$ is defined by
\begin{equation}
\Gamma_{\mathfrak{a}}:=\{\gamma \in \Gamma :  \gamma \mathfrak{a}=\mathfrak{a}\}.
\end{equation}

A scaling matrix for the cusp  $\mathfrak{a}$ is a matrix $\sigma_{ \mathfrak{a}}\in \mathbf{SL}_2(\R)$ such that
\begin{equation}
\sigma_{ \mathfrak{a}}\infty= \mathfrak{a}, \quad \sigma_{ \mathfrak{a}}^{-1}\Gamma_{\mathfrak{a}}\sigma_{ \mathfrak{a}}=\{\pm \begin{pmatrix}
1&n\\ 0&1
\end{pmatrix}:n \in \Z\}:=B.
\end{equation}
Note that the choice of scaling matrix is not unique.

Let $\chi$ be a Dirichlet character modulo $N$. This can be extended to $\Gamma$ as follows:
\begin{equation}
\chi \left(\gamma \right)=\chi(d), \quad \gamma=\begin{pmatrix}
a&b\\ cN&d
\end{pmatrix} \in \Gamma.
\end{equation}

Let $\lambda_{\mathfrak{a}}$ be defined by $\sigma_{\mathfrak{a}}^{-1}\lambda_{\mathfrak{a}}\sigma_{\mathfrak{a}}=\begin{pmatrix}1&1\\0&1\end{pmatrix}$.
The cusp $\mathfrak{a}$ is called singular for $\chi$ if $\chi(\lambda_{\mathfrak{a}})=1$. 

Suppose that $N=rs$, $(r,s)=1$. Then a cusp of the form $\mathfrak{a}=1/r$ is called an Atkin-Lehner cusp.
Note that Atkin-Lehner cusps are singular with respect to any Dirichlet character modulo $N$, see \cite[page 395]{KY}.


 Let $\kappa$ be defined by $\chi(-1)=(-1)^{\kappa}$ and let $\mathfrak{a}$, $\mathfrak{b}$ be two singular cusps  
for $\chi$ with corresponding scaling matrices  $\sigma_{\mathfrak{a}}$, $\sigma_{\mathfrak{b}}$.

Similarly to \cite[Eq. 2.3]{KY}, we define the Kloosterman sum associated to $\mathfrak{a}$, $\mathfrak{b}$ as
\begin{equation*}
S_{\mathfrak{a}\mathfrak{b}}(m,n;c;\chi):=\sum_{\gamma=\big(\begin{smallmatrix}
  a & b\\
  c & d
\end{smallmatrix}\big) \in \Gamma_{\infty}\setminus \sigma^{-1}_{\mathfrak{a}}\Gamma\sigma_{\mathfrak{b}}/\Gamma_{\infty}}\chi(\text{sgn}(c))
\overline{\chi(\sigma_{\mathfrak{a}}\gamma \sigma_{\mathfrak{b}}^{-1})}e \left( \frac{am+dn}{c}\right).
\end{equation*}
The set of allowed moduli is given by
\begin{equation}\label{eq:allowedmoduli}
 \it{C}_{\mathfrak{a},\mathfrak{b}}(N)=\{\gamma>0\text{ such that } \big(\begin{smallmatrix}
  * & *\\
  \gamma & *
\end{smallmatrix}\big)\in \sigma^{-1}_{\mathfrak{a}}\Gamma\sigma_{\mathfrak{b}}\}.
\end{equation}


\subsection{Holomorphic and Maass cusp forms}

Let $H_{k}(N,\chi)$ be an orthonormal basis of holomorphic cusp forms of weight $k>0,$ $k\equiv \kappa\pmod{2}$,  level $N$  and nebentypus $\chi$.
The Fourier expansion of $f \in H_{k}(N,\chi)$ around a singular cusp $\mathfrak{a}$  with a scaling matrix $\sigma_{\mathfrak{a}}$ is given by
\begin{equation}
f(\sigma_{\mathfrak{a}}z)i(\sigma_{\mathfrak{a}},z)^{-k}=\sum_{m\geq 1}\frac{\rho_{f_{\mathfrak{a}}}(m)}{\sqrt{m}}(4\pi m)^{k/2}e(mz),
\end{equation}
 where $i(\sigma_{\mathfrak{a}},z):=cz+d$ for  $\sigma_{\mathfrak{a}}=\begin{pmatrix}
*&*\\c&d
\end{pmatrix}$.

Let $H(N,\chi)$ be an orthonormal basis of the space of Maass cusp forms of weight $\kappa\in\{0,1\}$. For the function $f \in H(N,\chi)$ (which is an eigenfunction of the Laplace-Beltrami operator with eigenvalue $1/4+t_{f}^{2}$), the following Fourier-Whittaker expansion holds around the cusp $\mathfrak{a}$  with scaling matrix $\sigma_{\mathfrak{a}}$
\begin{equation}
f(\sigma_{\mathfrak{a}}z)e^{-i\kappa\arg i(\sigma_{\mathfrak{a}},z)}=\sum_{m\neq 0}\frac{\rho_{f_{\mathfrak{a}}}(m)}{\sqrt{m}}W_{\frac{|m|}{m}\frac{\kappa}{2},it_f}(4\pi|m|y)e(mx),
\end{equation}
where $z=x+iy$ and the Whittaker function $W_{\lambda,\mu}(z)$ is defined in \cite[Section 9.22]{GR}.

For  $f \in H_{k}(N,\chi)$ or $f \in H(N,\chi)$,  we define
\begin{equation}\label{L:sym}
L(s,\text{sym}^2 f_{\infty})=\zeta^{(N)}(2s)\sum_{l=1}^{\infty}\frac{\rho_{f_{\infty}}(l^2)}{l^s},\quad \Re{s}>1,
\end{equation}
where the superscript in $\zeta^{(N)}(2s)$ means that Euler factors at primes dividing $N$ have been removed.
Shimura \cite{S} proved an analytic continuation and a functional equation for \eqref{L:sym}.

\subsection{Eisenstein series}
Fix $\kappa=1$. For $\Gamma=\Gamma_0(N)$ the Eisenstein series associated to a singular cusp  $\mathfrak{c}$ for the nebentypus $\chi$ is defined as
\begin{equation}
E_{\mathfrak{c}}(z,s):=\sum_{\gamma \in \Gamma_{\mathfrak{c}}\setminus \Gamma}\overline{\chi}(\gamma)j_{\sigma_{\mathfrak{c}}^{-1}\gamma}(z)^{-1}
\left(\Im{(\sigma_{\mathfrak{c}}^{-1}\gamma z )}\right)^s,
\end{equation}
where $\sigma_{\mathfrak{c}}$ is a scaling matrix for $\mathfrak{c}$ and
\begin{equation}
j_{\gamma}(z):=\frac{cz+d}{|cz+d|}=e^{i\text{arg}(cz+d)}, \quad \gamma=\left(\begin{matrix}
  a & b\\
  c & d
\end{matrix}\right).
\end{equation}
\begin{thm}
Let $\mathfrak{c}$ be a singular cusp for the nebentypus $\chi$ and $\mathfrak{a}$ be the Aktin-Lehner cusp.
The following Fourier-Whittaker expansion holds
\begin{multline}
E_{\mathfrak{c}}(\sigma_{\mathfrak{a}}z,s)j_{\sigma_{\mathfrak{a}}}(z)^{-1}=\delta_{\mathfrak{a}\mathfrak{c}} y^s+\rho_{\mathfrak{a},\mathfrak{c}}(0,s)y^{1-s}\\+\sum_{m\neq 0}\rho_{\mathfrak{a},\mathfrak{c}}(m,s)e(mx)W_{\frac{|m|}{2m},s-1/2}(4\pi |m|y),
\end{multline}
where 
\begin{equation}
\rho_{\mathfrak{a}, \mathfrak{c}}(0,s)=-\frac{\sqrt{\pi}i\Gamma(s)}{\Gamma(s+1/2)}\phi_{\mathfrak{a},\mathfrak{c}}(0,s,\chi),
\end{equation}
\begin{equation}\label{eq:rhophi}
\rho_{\mathfrak{a}, \mathfrak{c}}(m,s)=-\frac{\pi^s i |m|^{s-1}}{\Gamma(s+1/2)}\phi_{\mathfrak{a},\mathfrak{c}}(m,s,\chi),
\end{equation}
\begin{multline}\label{eq:coeffFourierEisent}
\phi_{\mathfrak{a},\mathfrak{c}}(m,s,\chi)=\sum_{\gamma=\big(\begin{smallmatrix}
  * & *\\
  c & d
\end{smallmatrix}\big)\in \Gamma_{\infty} \setminus \sigma_{\mathfrak{c}}^{-1}\Gamma\sigma_{\mathfrak{a}}/\Gamma_{\infty}}\overline{\chi}(\sigma_{\mathfrak{c}}\gamma\sigma_{\mathfrak{a}}^{-1})\frac{e(md/c)}{c^{2s}}\\
=\sum_{c\in C_{\mathfrak{c,a}}(N)}\frac{S_{\mathfrak{ca}}(0,m;c;\chi)}{c^{2s}}.
\end{multline}
\end{thm}
\begin{proof}
Consider the Eisenstein series
\begin{equation}
E_{\mathfrak{c}}(\sigma_{\mathfrak{a}}z,s)=\sum_{\gamma \in \Gamma_{\mathfrak{c}}\setminus \Gamma}\overline{\chi}(\gamma)j_{\sigma_{\mathfrak{c}}^{-1}\gamma}(\sigma_{\mathfrak{a}}z)^{-1}\left(\Im{(\sigma_{\mathfrak{c}}^{-1}\gamma\sigma_{\mathfrak{a}}z)}  \right)^s.
\end{equation}

Making the change of variables $\tau:=\sigma_{\mathfrak{c}}^{-1}\gamma\sigma_{\mathfrak{a}}$ (so that $\tau \in B\setminus \sigma_{\mathfrak{c}}^{-1}\Gamma\sigma_{\mathfrak{a}}$) we infer
\begin{equation}
E_{\mathfrak{c}}(\sigma_{\mathfrak{a}}z,s)=\sum_{\tau \in B\setminus \sigma_{\mathfrak{c}}^{-1}\Gamma\sigma_{\mathfrak{a}}}\overline{\chi}(\sigma_{\mathfrak{c}}\tau\sigma_{\mathfrak{a}}^{-1})j_{\tau\sigma_{\mathfrak{a}}^{-1}}(\sigma_{\mathfrak{a}}z)^{-1}\left(\Im{(\tau z)}  \right)^s.
\end{equation}
Using the property
\begin{equation}
j_{\tau\sigma_{\mathfrak{a}}^{-1}}(\sigma_{\mathfrak{a}}z)^{-1}j_{\sigma_{\mathfrak{a}}}(z)^{-1}=j_{\tau}(z)^{-1}
\end{equation}
we find that
\begin{multline}
E_{\mathfrak{c}}(\sigma_{\mathfrak{a}}z,s)j_{\sigma_{\mathfrak{a}}}(z)^{-1}=\sum_{\tau \in B\setminus \sigma_{\mathfrak{c}}^{-1}\Gamma\sigma_{\mathfrak{a}}}\overline{\chi}(\sigma_{\mathfrak{c}}\tau\sigma_{\mathfrak{a}}^{-1})j_{\tau}(z)^{-1}\left(\Im{(\tau z)}  \right)^s\\=\delta_{\mathfrak{a}\mathfrak{c}} y^s+\sum_{\gamma \in B\setminus \sigma_{\mathfrak{c}}^{-1}\Gamma\sigma_{\mathfrak{a}}/B}\sum_{\tau \in B}\overline{\chi}(\sigma_{\mathfrak{c}}\gamma \tau\sigma_{\mathfrak{a}}^{-1})j_{\gamma\tau}(z)^{-1}\left(\Im{(\gamma \tau z)}  \right)^s.
\end{multline}
Note that $\overline{\chi}(\sigma_{\mathfrak{c}}\gamma \tau\sigma_{\mathfrak{a}}^{-1})=\overline{\chi}(\sigma_{\mathfrak{c}}\gamma \sigma_{\mathfrak{a}}^{-1})$ since $\mathfrak{a}$ is singular. Furthermore,
taking $\gamma=\begin{pmatrix}a&b\\c&d\end{pmatrix}$ and $\tau=\begin{pmatrix}1&n\\0&1\end{pmatrix}$, we obtain
\begin{equation}
\gamma\tau z=\frac{a}{c}-\frac{1}{c(c(z+n)+d)}, \quad j_{\gamma \tau}=\frac{cz+cn+d}{|cz+cn+d|}.
\end{equation}
Consequently, for $z=x+iy$ we have
\begin{equation}
\Im{(\gamma \tau z)}=\frac{y}{c^2}\frac{1}{(x+n+d/c)^2+y^2}
\end{equation}
and
\begin{multline}\label{eq:Eisenstein}
E_{\mathfrak{c}}(\sigma_{\mathfrak{a}}z,s)j_{\sigma_{\mathfrak{a}}}(z)^{-1}=\delta_{\mathfrak{a}\mathfrak{c}} y^s+\sum_{\gamma \in B\setminus \sigma_{\mathfrak{c}}^{-1}\Gamma\sigma_{\mathfrak{a}}/B}\overline{\chi}(\sigma_{\mathfrak{c}}\gamma \sigma_{\mathfrak{a}}^{-1})\\ \times \sum_{n\in \Z}\left(\frac{cz+cn+d}{|cz+cn+d|} \right)^{-1}\left(\frac{y}{c^2}\frac{1}{(x+n+d/c)^2+y^2}  \right)^s.
\end{multline}

In order to evaluate the sum over $n$ we apply the Poisson summation formula, showing that
\begin{multline}
\sum_{n\in \Z}\left(\frac{cz+cn+d}{|cz+cn+d|} \right)^{-1}\left(\frac{y}{c^2}\frac{1}{(x+n+d/c)^2+y^2}  \right)^s\\=
\sum_{m\in \Z}\int_{-\infty}^{\infty} \left( \frac{c(z+v)+d}{|c(z+v)+d|}\right)^{-1} \frac{(yc^{-2})^se(-mv)}{((x+d/c+v)+y^2)^s}dv.
\end{multline}

Making the change of variables $t:=x+d/c+v$, this is equal to
\begin{multline}\label{eq:afterPoisson}
\sum_{m\in \Z} e\left(mx+\frac{md}{c}\right)\int_{-\infty}^{\infty}\left(\frac{t+iy}{|t+iy|} \right)^{-1}\left( \frac{yc^{-2}}{t^2+y^2}\right)^se(-mt)dt.
\end{multline}

Let us assume first that $m=0$. Then 
\begin{multline}
\int_{-\infty}^{\infty}\frac{|t+iy|}{t+iy}\frac{(yc^{-2})^s}{(t^2+y^2)^s}dt=\frac{1}{i}\int_{-\infty}^{\infty}\frac{(yc^{-2})^sdt}{(y+it)^{s-1/2}(y-it)^{s+1/2}}\\
=\frac{(yc^{-2})^s}{i}\frac{2\pi (2y)^{1-2s}\Gamma(2s)}{(2s-1)\Gamma(s-1/2)\Gamma(s+1/2)}=-\frac{\sqrt{\pi}i \Gamma(s)}{\Gamma(s+1/2)}\frac{y^{1-s}}{c^{2s}},
\end{multline}
where we used \cite[Eq. 8.381.1]{GR} to evaluate the integral.
If $m \neq 0$ the integral in \eqref{eq:afterPoisson} can be computed using \cite[3.384.9]{GR} as follows
\begin{multline}
\frac{(yc^{-2})^s}{i}\int_{-\infty}^{\infty}(y+it)^{1/2-s}(y-it)^{-1/2-s}e(-mt)dt\\=\frac{(2\pi)^s2^{-s}|m|^{s-1}}{ic^{2s}\Gamma(1/2+s)}W_{\frac{|m|}{2m},1/2-s}(4\pi y|m|).
\end{multline}
Consequently, \eqref{eq:afterPoisson} is equal to
\begin{multline}\label{eq:sumoverm}
-\frac{\pi^{1/2}i\Gamma(s)}{\Gamma(s+1/2)}y^{1-s}\frac{1}{c^{2s}}\\+\frac{\pi^s}{i\Gamma(1/2+s)}\frac{1}{c^{2s}}\sum_{m\neq 0}|m|^{s-1}e(mx+md/c)W_{\frac{|m|}{2m},1/2-s}(4\pi y|m|).
\end{multline}
The statement follows by replacing the sum over $n$ in \eqref{eq:Eisenstein} by \eqref{eq:sumoverm}.
\end{proof}

\subsection{Kuznetsov trace formula}
In this section we state the Kuznetsov trace formula for Dirichlet multiplier system and general cusps.
To this end, we follow \cite{DI}, \cite[Section 3.3]{KY2} and \cite[Section 4.1.3]{Dra} assuming that $\kappa=1$ (i.e. $\chi(1)=-1$).

Consider the function $\psi \in C^{\infty}$ such that 
\begin{equation}
\psi(0)=\psi'(0)=0, \quad \psi^{(j)}(x)\ll (1+x)^{-2-\eta}, \quad j=0,1,2,3
\end{equation}
for some $\eta>0$.
It is also required to introduce the following transforms:
\begin{equation}\label{eq:psiHk}
\psi_H(k):=4i^k\int_{0}^{\infty}J_{k-1}(x)\psi(x)\frac{dx}{x},
\end{equation}
\begin{equation}\label{eq:psiDt}
\psi_D(t):=\frac{2\pi i t}{\sinh{(\pi t)}}\int_{0}^{\infty}(J_{2it}(x)+J_{-2it}(x))\psi(x)\frac{dx}{x},
\end{equation}
where $J_{\alpha}$  denotes the $J$-Bessel function of order $\alpha$.

For $m,n \geq 1$
\begin{equation}\label{eq:Kuznetsov}
\sum_{c\in C_{\mathfrak{a},\mathfrak{b}}(N)}\frac{S_{\mathfrak{a}\mathfrak{b}}(m,n;c;\chi)}{c}\psi\left(\frac{4\pi\sqrt{mn}}{c} \right)=
H+D+C,
\end{equation}
where
\begin{equation}
H:=\sum_{\substack{k>1\\ k\equiv 1 \pmod{2}}}\sum_{f \in H_k(N,\chi)}\psi_H(k)\Gamma(k)
\overline{\rho_{f_{\mathfrak{a}}}(m)}\rho_{f_{\mathfrak{b}}}(n),
\end{equation}
\begin{equation}
D:=\sum_{f\in H(N,\chi)}\frac{\psi_D(t_f)}{\cosh{(\pi t_f)}}\overline{\rho_{f_{\mathfrak{a}}}(m)}\rho_{f_{\mathfrak{b}}}(n),
\end{equation}
\begin{equation}
C:=\sum_{\mathfrak{c} \text{ sing.}}\frac{\sqrt{mn}}{4\pi}\int_{-\infty}^{\infty}\frac{\psi_D(t)}{\cosh{(\pi t)}}\overline{\rho_{\mathfrak{a,c}}(m,1/2+it)}\rho_{\mathfrak{b,c}}(n,1/2+it)dt.
\end{equation}

According to \eqref{eq:rhophi} the continuous part can be written as
\begin{multline}
C=\sum_{\mathfrak{c} \text{ sing.}}\frac{1}{4\pi}\int_{-\infty}^{\infty}\frac{\psi_D(t)}{\cosh{(\pi t)}}\frac{\pi }{|\Gamma(1+it)|^2}m^{-it}n^{it}
\\ \times \overline{\phi_{\mathfrak{a,c}}(m,1/2+it,\chi)}\phi_{\mathfrak{b,c}}(n,1/2+it,\chi)dt.
\end{multline}

Using  the identity (see \cite[Eq. 5.4.3]{HMF})
\begin{equation}
\Gamma(1+it)\Gamma(1-it)=\frac{\pi t}{\sinh(\pi t)}
\end{equation}
we find that
\begin{multline}\label{spectra:cont}
C=
\sum_{\mathfrak{c} \text{ sing.}}\frac{1}{4\pi}\int_{-\infty}^{\infty}\frac{\psi_D(t)\sinh(\pi t)}{t\cosh{(\pi t)}}m^{-it}n^{it}\\ \times \overline{\phi_{\mathfrak{a,c}}(m,1/2+it,\chi)}\phi_{\mathfrak{b,c}}(n,1/2+it,\chi)dt.
\end{multline}

\subsection{Gauss sums}\label{section:gauss}

For a  Dirichlet character $\chi$ modulo $q$, we define the Gauss sum of $\chi$ by
\begin{equation}\label{eq:miyake}
g(\chi;q;m):=\sum_{\substack{u\pmod{q}\\(u,q)=1}}\chi(u)e\left(\frac{mu}{q} \right), \quad \tau(\chi):=g(\chi;q;1).
\end{equation}

 Let $\chi$ be a character modulo $q$ induced from a primitive character $\chi^*$ modulo $q^*$. Then
according to \cite[Lemma 3.1.3 (2)]{Miy} we have
\begin{equation}\label{eq:miyake2}
g(\chi;q;m)=\tau(\chi^*)\sum_{d|(m,q/q^*)}d\chi^*\left( \frac{q}{q^*d}\right)\overline{\chi^*}\left( \frac{m}{d}\right)\mu\left(\frac{q}{q^{*}d}\right).
\end{equation}

The generalized quadratic Gauss sums is given by
\begin{equation}
G(a,n;q):=\sum_{x\pmod{q}}e\left(\frac{ax^2+nx}{q} \right), \quad (a,q)=1.
\end{equation}
Let $G(a;q):=G(a,0;q)$.

The notation $\overline{a}_q$ means that $\overline{a}_qa\equiv 1\pmod{q}$.

\begin{lem}\label{lem:g^2} For $(a,q)=1$ the following identity holds
\begin{equation}
G^2(a;q)=\begin{cases}
q\chi_4(q), \quad q \text{ is odd}\\
0, \quad q\equiv 2\pmod{4}\\
2qi\chi_4(a),\quad q\equiv 0\pmod{4}.
\end{cases}
\end{equation}
\end{lem}
\begin{proof}
If $q$ is odd we have by \cite[Eq. 23]{Mal}
\begin{equation}
G^2(a;q)=\left(\left( \frac{a}{q}\right)i^{(\frac{q-1}{2})^2}\sqrt{q} \right)^2=q(-1)^{(\frac{a-1}{2})^2}=q\chi_4(q).
\end{equation}
If $q\equiv 2\pmod{4}$, then using \cite[Eq. 25]{Mal} we find that $G(a;q)=0$. Finally, if $q\equiv 0\pmod{4}$, then \cite[Eq. 25]{Mal} implies that
\begin{equation}
G^2(a;q)=q(1+i^a)^2=2qi^a=2qi\chi_4(a).
\end{equation}
\end{proof}

\begin{lem} If $n$ is even and $q$ is odd, then 
\begin{equation}\label{eq:gfornevenqodd}
G(a,n;q)=e\left(-\frac{\overline{a}_q(n/2)^2}{q} \right)G(a;q)=e\left(-\frac{\overline{(4a)}_qn^2}{q} \right)G(a;q).
\end{equation}
\end{lem}
\begin{proof}
See \cite[Eq. (26)]{Mal}.
\end{proof}

\begin{lem}\label{lem:neven_and_qeven}  Suppose that $n$ and $q$ are even. 

If $q\equiv2\pmod{4}$, then $G(a,n;q)=0$.

If $q\equiv0\pmod{4}$, then
\begin{equation}\label{eq:Mal25}
G(a,n;q)=e\left(-\frac{\overline{a}_q(n/2)^2}{q} \right)G(a;q).
\end{equation}
\end{lem}
\begin{proof}
The relation \eqref{eq:Mal25} follows from \cite[Eq. (26)]{Mal} for any even $n$ and $q$.
However, $G(a;q)=0$ for $q\equiv2\pmod{4}$ by \cite[Eq. (25)]{Mal}.

\end{proof}

\begin{lem} Suppose that $n$ and $q$ are odd. Then 
\begin{equation}\label{eq:nqodd}
G(a,n;q)=e\left(-\frac{\overline{(4a)}_qn^2}{q} \right)G(a;q).
\end{equation}
\end{lem}
\begin{proof}
In this case according to \cite[Eq. (27)]{Mal} we have
\begin{equation}
G(a,n;q)=e\left(-\frac{\overline{a}_q(\frac{n+q}{2})^2}{q} \right)G(a;q).
\end{equation}
Note that it follows from the relation
\begin{equation}
n^2\equiv(n+q)^2\pmod{q}
\end{equation}
that
\begin{equation}
\overline{(4a)}_qn^2\equiv  \overline{a}_q\left(\frac{n+q}{2}\right)^2\pmod{q}.
\end{equation}
This implies the statement.
\end{proof}

\begin{lem}\label{lem:nodd_and_qeven} 
Suppose that $n$ is odd and $q$ is even.  If $q\equiv 0\pmod{4}$, then $G(a,n;q)=0$.
Otherwise, we can write $q=2r$ with $r$ odd. 
\begin{equation}\label{eq:8a}
G(a,n;q)=2e\left(-\frac{\overline{(8a)}_rn^2}{r} \right)G(2a;r).
\end{equation}
\end{lem}
\begin{proof}
If $n$ is odd and $q\equiv 0\pmod{4}$, then $G(a,n;q)=0$ by \cite[Eq. (28)]{Mal}.
Assume that $q=2r$ with $r$ odd. Using the twisted multiplicativity of Gauss sums, we find
\begin{equation}
G(a,n;2r)=G(a\overline{2}_r,n\overline{2}_r;r)G(a\overline{r}_2,n\overline{r}_2;2).
\end{equation}
By direct computations $G(a\overline{r}_2,n\overline{r}_2;2)=2$. Finally,
\begin{equation}
G(a\overline{2}_r,n\overline{2}_r;r)=G(2a,n;r)=e\left(-\frac{\overline{(8a)}_rn^2}{r} \right)G(2a;r),
\end{equation}
where we used \eqref{eq:nqodd} to evaluate $G(2a,n;r)$.
\end{proof}

\section{Diagonal and non-diagonal terms}\label{sect:diagnondiag}
Assuming that $s$ is sufficiently large, we prove in this section an explicit formula for $$\sum_{l=1}^{\infty}\omega(l)\mathscr{L}_{n^2-4l^2}(s)$$ with diagonal and non-diagonal terms.
Applying the results of Section \ref{section:gauss}, we compute the diagonal term explicitly and prove an  expression for the non-diagonal term in terms of sums of Kloosterman sums
suitable for application of the Kuznetsov trace formula.

\begin{lem}
For $\Re{s}>3/2$ the following formula holds
\begin{multline}\label{eq:mainform}
\sum_{l=1}^{\infty}\omega(l)\mathscr{L}_{n^2-4l^2}(s)=\\ \widehat{\omega}(1)\zeta(2s)\sum_{q=1}^{\infty}\frac{1}{q^{2+s}}\sum_{c,d\pmod{q}}S(d^2,c^2;q)e\left(\frac{nc}{q} \right)\\+
\frac{\zeta(2s)}{\pi^{s-1/2}}\sum_{l=1}^{\infty}\frac{1}{l^s}\sum_{q=1}^{\infty}\frac{f\left( \omega,s;4\pi n l/q\right)}{q^{2}}
\sum_{c,d\pmod{q}}S(d^2,c^2;q)e\left( \frac{nc+ld}{q}\right),
\end{multline}
where $S(d,c;q)$ is the ordinary Kloosterman sum and for $a<1$
\begin{equation}\label{eq:fwsx}
f(\omega,s;x):=\frac{1}{2\pi i}\int_{(a)}\widehat{\omega}(\alpha)\frac{\Gamma(1/2-\alpha/2)}{\Gamma(\alpha/2)}\left( \frac{x}{4n}\right)^{\alpha+s-1}d\alpha.
\end{equation}
\end{lem}

\begin{proof}
Using the Mellin transform for the function $\omega$ we have
\begin{equation}\label{eq:sumwithomega}
\sum_{l=1}^{\infty}\omega(l)\mathscr{L}_{n^2-4l^2}(s)=\frac{1}{2\pi i}\int_{(a)}\widehat{\omega}(\alpha)\sum_{l=1}^{\infty}\frac{\mathscr{L}_{n^2-4l^2}(s)}{l^{\alpha}}d\alpha,
\end{equation}
where $a>1$.
According to \cite[Eq. 4-9]{BF1}, the function $\mathscr{L}_{n^2-4l^2}(s)$ can be written in terms of sums of Kloosterman sums
\begin{equation}
\mathscr{L}_{n^2-4l^2}(s)=\zeta(2s)\sum_{q=1}^{\infty}\frac{1}{q^{1+s}}\sum_{c\pmod{q}}S(l^2,c^2;q)e\left(\frac{nc}{q} \right).
\end{equation}

 Substituting this expression into \eqref{eq:sumwithomega} we can change the order of summation, summing with respect to $q$ first and over $l$ second, as long as $\Re{s}>3/2$. Furthermore, dividing the range of summation for $l$ into arithmetic progressions, we obtain
\begin{multline}\label{eq:sumwithomega2}
\sum_{l=1}^{\infty}\omega(l)\mathscr{L}_{n^2-4l^2}(s)=\frac{\zeta(2s)}{2\pi i}\int_{(a)}\widehat{\omega}(\alpha)\sum_{q=1}^{\infty}\frac{1}{q^{1+s}}\\ \times \sum_{c,d\pmod{q}}S(d^2,c^2;q)e\left( \frac{nc}{q}\right)
\sum_{\substack{l\geq 1\\ l\equiv d\pmod{q}}}\frac{1}{l^{\alpha}}d\alpha.
\end{multline}

The inner sum on the right-hand side of \eqref{eq:sumwithomega2} can be written in terms of the Lerch zeta function $\zeta(a,b;\alpha)$ with the following functional equation
\begin{multline}\label{funceq Lerch}
\sum_{\substack{l\geq 1\\ l\equiv d\pmod{q}}}\frac{1}{l^{\alpha}}=\frac{1}{q^{\alpha}}\sum_{l=1}^{\infty}\frac{1}{(l+d/q)^{\alpha}}=\frac{\zeta(d/q,0;\alpha)}{q^{\alpha}}=\\
\frac{\Gamma(1-\alpha)}{q^{\alpha}(2\pi)^{1-\alpha}}\left[-ie(\alpha/4)\sum_{l=1}^{\infty}\frac{e(ld/q)}{l^{1-\alpha}}+ie(-\alpha/4)\sum_{l=1}^{\infty}\frac{e(-ld/q)}{l^{1-\alpha}}\right].
\end{multline}
In order to apply this functional equation, we move the contour of integration in \eqref{eq:sumwithomega2} to $\Re{\alpha}<0$, crossing a simple pole of $\zeta(d/q,0;\alpha)$ at $\alpha=1$.
The contribution of this pole is equal to
\begin{equation}\label{eq:poleContr}
\widehat{\omega}(1)\zeta(2s)\sum_{q=1}^{\infty}\frac{1}{q^{2+s}}\sum_{c,d\pmod{q}}S(d^2,c^2;q)e\left(\frac{nc}{q} \right).
\end{equation}
Next, applying the functional equation \eqref{funceq Lerch} and using the fact that
\begin{equation}
ie(-\alpha/4)-ie(\alpha/4)=2\sin\left(\frac{\pi \alpha}{2}\right),
\end{equation}
we find that \eqref{eq:sumwithomega2}  is equal to \eqref{eq:poleContr} plus
\begin{multline}\label{eq:nonpoleContr}
\frac{\zeta(2s)}{2\pi i}\int_{(a)}\widehat{\omega}(\alpha)\sum_{q=1}^{\infty}\frac{1}{q^{1+s+\alpha}}\sum_{l=1}^{\infty}\frac{1}{l^{1-\alpha}}\\ \times \sum_{c,d\pmod{q}}S(d^2,c^2;q)e\left( \frac{nc+ld}{q}\right)\frac{2\Gamma(1-\alpha)\sin(\pi \alpha/2)}{(2\pi)^{1-\alpha}}
d\alpha, \quad a<0.
\end{multline}

It follows from \cite[Eqs. 5.5.5, 5.5.3]{HMF} that
\begin{equation}
\Gamma(1-\alpha)\sin(\pi \alpha/2)=\pi^{1/2}2^{-\alpha}\frac{\Gamma(1/2-\alpha/2)}{\Gamma(\alpha/2)}.
\end{equation}
Substituting this into \eqref{eq:nonpoleContr}, we prove the lemma.
\end{proof}

\subsection{The inner sum}
In order to evaluate the inner sum in \eqref{eq:mainform} we express it as a sum of a product of two Gauss sums:
\begin{multline}\label{eq:ksum}
K(n,l;q):=\sum_{c,d\pmod{q}}S(c^2,d^2,q)e\left( \frac{nc+ld}{q}\right)\\=\sum_{\substack{a,b\pmod{q}\\ab\equiv 1\pmod{q}}}\sum_{c,d\pmod{q}}e\left(\frac{ac^2+db^2}{q} \right)
e\left(\frac{nc+ld}{q} \right)\\=\sum_{\substack{a,b\pmod{q}\\ab\equiv 1\pmod{q}}}G(a,n;q)G(b,l;q).
\end{multline}

It is required to examine different cases depending on the even-odd parity of the parameters $n$, $l$, and $q$. To this end, we use the results of Section \ref{section:gauss}. 

\begin{lem}\label{lem:supqisodd}
Suppose that $q$ is odd. Then 
\begin{equation}
K(n,l;q)=q\chi_4(q)\sum_{\substack{a,b\pmod{q}\\ab\equiv1\pmod{q}}}e\left(-\frac{\overline{4}_q bn^2+\overline{4}_q al^2}{q} \right).
\end{equation}
\end{lem}
\begin{proof}
There are four different cases to consider depending on the parity of $n$ and $l$.
Assume first that $n$ and $l$ are even.
Using \eqref{eq:ksum} and \eqref{eq:gfornevenqodd}, we have
\begin{equation}
K(n,l;q)=\sum_{\substack{a,b\pmod{q}\\ab\equiv1\pmod{q}}}e\left(-\frac{\overline{4}_q bn^2+\overline{4}_q al^2}{q} \right)G(a;q)G(b;q).
\end{equation}
Note that $G(a;q)=G(b;q)$ since $ab\equiv 1\pmod{q}$. Then Lemma \ref{lem:g^2} implies the statement when $n$ and $l$ are even.
Other three cases can be treated similarly by evaluating $G(a,n;q)$ using  \eqref{eq:gfornevenqodd} if $n$ is even and \eqref{eq:nqodd} if $n$ is odd.
\end{proof}

\begin{lem}\label{lem:supqiseven}
Suppose that $q$ is even and $n+l$ is odd. Then 
\begin{equation}
K(n,l;q)=0.
\end{equation}
\end{lem}
\begin{proof}
According to \eqref{eq:ksum} we have
\begin{equation}
K(n,l;q)=\sum_{\substack{a,b\pmod{q}\\ab\equiv 1\pmod{q}}}G(a,n;q)G(b,l;q),
\end{equation}
where one of the parameters $n$ and $l$ is even and another one is odd. Without loss of generality, we can assume that  $n$ is even and $l$ is odd.
Then Lemma \ref{lem:neven_and_qeven} implies that $G(a,n;q)$ is nonzero only if $q\equiv0\pmod{4}$, but in that case $G(b,l;q)=0$ by Lemma \ref{lem:nodd_and_qeven}.
\end{proof}

\begin{lem}\label{lem:qnleven}
Suppose that $q$, $n$ and $l$ are even. Then $K(n,l;q)=0$ if $q\equiv 2\pmod{4}$, and otherwise
\begin{equation}\label{eq:alleven}
K(n,l;q)=2iq\sum_{\substack{a,b\pmod{q}\\ab\equiv1\pmod{q}}}\chi_4(a)e\left(-\frac{a(l/2)^2+b(n/2)^2}{q} \right).
\end{equation}
\end{lem}
\begin{proof}
It follows directly from  \eqref{eq:ksum} and Lemma \ref{lem:neven_and_qeven} that $K(n,l;q)\neq 0$ only if $q\equiv 0\pmod{4}$.
In that case, \eqref{eq:alleven} is a consequence of Lemmas \ref{lem:neven_and_qeven} and \ref{lem:g^2}.
\end{proof}

\begin{lem}\label{lem:qevennlodd}
Suppose that $q$ is even and $n$, $l$ are odd. Then $K(n,l;q)=0$ if  $q\equiv 0\pmod{4}$.   If $q\equiv 2\pmod{4}$, then $r:=q/2$ is odd and
\begin{equation}
K(n,l;q)=2q\chi_4(r)S\left(\overline{(8)}_r n^2,\overline{(8)}_r l^2;r\right).
\end{equation}
\end{lem}
\begin{proof}
If  $q\equiv 0\pmod{4}$, then \eqref{eq:ksum} and Lemma \ref{lem:nodd_and_qeven} imply that $K(n,l;q)=0$.
In the opposite case, it follows from \eqref{eq:8a} and Lemma \ref{lem:g^2} that
\begin{equation}
K(n,l;q)=2q\chi_4(r)\sum_{\substack{a,b\pmod{q}\\ab\equiv1\pmod{q}}}e\left(-\frac{\overline{(8a)}_r n^2+\overline{(8b)}_r l^2}{r} \right),
\end{equation}
where  $r:=q/2$ is odd.
This is equal to
\begin{equation}
K(n,l;q)=2q\chi_4(r)S(2\overline{(8)}_r n^2,2\overline{(8)}_r l^2;q).
\end{equation}
Writing $q=2r$ and applying the multiplicity property of Kloosterman sums, we infer
\begin{equation*}
K(n,l;q)=2q\chi_4(r)S\left(\overline{(8)}_r n^2,\overline{(8)}_r l^2;r\right)S\left(2\overline{(8)}_r \overline{(r)}_2 n^2,2\overline{(8)}_r \overline{(r)}_2 l^2;2\right).
\end{equation*}
Then the assertion follows by using the fact that
$$S\left(2\overline{(8)}_r \overline{(r)}_2 n^2,2\overline{(8)}_r \overline{(r)}_2 l^2;2\right)=1.$$
\end{proof}



\subsection{The diagonal term}
Now we are ready to evaluate the diagonal main term in \eqref{eq:mainform}, namely
\begin{equation}\label{eq:md}
M^D(n,s):=\widehat{\omega}(1)\zeta(2s)\sum_{q=1}^{\infty}\frac{1}{q^{2+s}}K(n,0;q).
\end{equation}

We use the subscripts "even" and "odd" in $M^{D}_{\text{even}}(n,s)$ and $M^{D}_{\text{odd}}(n,s)$ to mark the parity of $n$.
\begin{lem}\label{lem:diagonal}
If $n$ is even, then
\begin{equation}\label{eq:meven}
M^{D}_{\text{even}}(n,s)=\frac{\widehat{\omega}(1)\zeta(2s)}{L(\chi_4,1+s)}\left(n^{-2s}\sigma_s\left(\chi_4;n^2\right)+\sigma_{-s}(\chi_4;n^2)\right).
\end{equation}
If $n$ is odd, then
\begin{equation}\label{eq:mdodd}
M^{D}_{\text{odd}}(n,s)=\frac{\widehat{\omega}(1)\zeta(2s)}{L(\chi_4,1+s)}\sigma_{-s}(\chi_4;n^2).
\end{equation}
\end{lem}
\begin{proof}
Assume first that $n$ is even. We can split the the sum over $q$ in \eqref{eq:md} as follows
\begin{equation}\label{sum:splitoddeven}
\sum_{q=1}^{\infty}=\sum_{q\equiv 0\pmod{2}}+\sum_{q\equiv 1\pmod{2}}.
\end{equation}
For $n$ and $q$ even, we have $K(n,0;q)=0$  unless $q\equiv 0\pmod{4}$. If $q\equiv 0\pmod{4}$  the following identity holds $\chi_4(a)=\chi_4(b)$ for $ab\equiv 1\pmod{q}$. 
Then by Lemma \ref{lem:qnleven}
\begin{equation}
\frac{K(n,0;q)}{2iq}=\sum_{b\pmod{q}}^{*}\chi_4(b)e\left(-\frac{b(n/2)^2}{q} \right)=\chi_4(-1)g\left(\chi_4;q;(n/2)^2\right),
\end{equation}
where the star over the sum above means that we are summing over $b\pmod{q}$ such that $(b,q)=1$.
Consequently, using \eqref{eq:miyake2} we find that
\begin{multline}
\sum_{q \equiv 0\pmod{2}}\frac{K(n,0;q)}{q^{2+s}}=2i\chi_4(-1)\sum_{q\equiv 0\pmod{4}}\frac{g(\chi_4;q;(n/2)^2)}{q^{1+s}}\\
=\frac{2i\chi_4(-1)}{4^{1+s}}\tau(\chi_4)\sum_{q=1}^{\infty}\frac{1}{q^{1+s}}\sum_{d|((n/2)^2,q)}d\chi_4\left(\frac{q}{d}\right)\chi_4\left(\frac{(n/2)^2}{d}\right)\mu\left(\frac{q}{d}\right).
\end{multline}

Computing the sum over $q$, we obtain
\begin{multline}\label{case:even}
\sum_{q \equiv 0\pmod{2}}\frac{K(n,0;q)}{q^{2+s}}=\frac{2i\chi_4(-1)\tau(\chi_4)}{4^{1+s}L(\chi_4,1+s)}\sum_{d|(n/2)^2}d^{-s}\chi_4\left(\frac{(n/2)^2}{d}\right)\\
=\frac{2i\chi_4(-1)\tau(\chi_4)}{4^{1+s}L(\chi_4,1+s)}\left(\frac{n}{2} \right)^{-2s}\sigma_s\left(\chi_4;(n/2)^2\right)=\frac{\sigma_s\left(\chi_4;(n/2)^2\right)}{L(\chi_4,1+s)}n^{-2s}.
\end{multline}

Now consider the sum over odd $q$ in \eqref{sum:splitoddeven}.
Applying Lemma \ref{lem:supqisodd} to compute $K(n,0;q)$  and making the change of variables $-\overline{4}_qb\rightarrow b$, we evaluate the second sum
\begin{multline}\label{case:odd}
\sum_{q \equiv 1\pmod{2}}\frac{K(n,0;q)}{q^{2+s}}=\sum_{q=1}^{\infty}\frac{\chi_4(q)}{q^{1+s}}\sum_{b\pmod{q}}e\left(\frac{bn^2}{q} \right)\\=\sum_{q=1}^{\infty}\frac{\chi_4(q)}{q^{1+s}}\sum_{d|(q,n^2)}d\mu\left(\frac{q}{d} \right)\\
=\sum_{d|n^2}d\sum_{q\equiv 0\pmod{d}}\frac{\chi_4(q)\mu(q/d)}{q^{1+s}}=\frac{\sigma_{-s}(\chi_4;n^2)}{L(\chi_4,1+s)}.
\end{multline}

Combining \eqref{case:odd} and \eqref{case:even} and using the fact that  $\sigma_s\left(\chi_4;(n/2)^2\right)=\sigma_s(\chi_4;n^2)$,  we prove \eqref{eq:meven}.

Similarly, for odd $n$ the identity \eqref{eq:mdodd} follows from
\begin{equation}\label{case:oddodd}
\sum_{q=1}^{\infty}\frac{K(n,0;q)}{q^{2+s}}=\frac{\sigma_{-s}(\chi_4;n^2)}{L(\chi_4,1+s)}.
\end{equation}
\end{proof}

\subsection{The non-diagonal term}
In this subsection, we study the non-diagonal term in \eqref{eq:mainform}, namely
\begin{equation}
M^{ND}:=\frac{\zeta(2s)}{\pi^{s-1/2}}\sum_{l=1}^{\infty}\frac{1}{l^s}\sum_{q=1}^{\infty}\frac{f\left( \omega,s;4\pi n l/q\right)}{q^{2}}
K(n,l;q).
\end{equation}
For simplicity, let $\psi(x):=f\left( \omega,s;4x\right)$. 
Consider the following two Kloosterman sums (see, for example, \cite[Eqs. (2.20), (2.23)]{KY}) :
\begin{equation}\label{eq:sinf0}
S_{\infty,0}(m,n;c\sqrt{N};\chi)=\overline{\chi}(c)S(\overline{N}m,n;c), \quad (c,N)=1,
\end{equation}
\begin{equation}
S_{\infty,\infty}(m,n;c;\chi)=\sum_{ab\equiv 1\pmod{c}}e\left(\frac{am+bn}{c} \right)\overline{\chi}(b).
\end{equation}

Using \eqref{eq:allowedmoduli} and \cite[Eq. 2.15]{KY} we find that
\begin{equation}
C_{\infty,\infty}(4)=\{\gamma=q> 0, \quad q\equiv 0\pmod{4}\},
\end{equation}
\begin{equation}
C_{\infty,0}(4)=\{ \gamma=2q> 0, (q,4)=1\},
\end{equation}
\begin{equation}
C_{\infty,0}(16)=\{\gamma=4q> 0,\quad (q,2)=1\},
\end{equation}
\begin{equation}
C_{\infty,0}(64)=\{\gamma=8s> 0,\quad (s,2)=1\}.
\end{equation}

\begin{lem}
If $n$ is even (let $n_1:=n/2$), then
\begin{multline}\label{eq:mndeven}
M^{ND}=-\frac{2i\zeta(2s)}{2^s\pi^{s-1/2}}\sum_{l=1}^{\infty}\frac{1}{l^s}\sum_{\gamma \in C_{\infty, \infty}(4)}\frac{1}{\gamma}\psi\left(\frac{4\pi ln_1}{\gamma} \right)S_{\infty\infty}\left(l^2,n_{1}^{2};\gamma;\chi_4\right)\\
+\frac{2\zeta(2s)}{\pi^{s-1/2}}\sum_{l=1}^{\infty}\frac{1}{l^s}\sum_{\gamma\in C_{\infty, 0}(4)}\frac{1}{\gamma}\psi\left(\frac{4\pi ln_1}{\gamma} \right)S_{\infty 0}(l^2,n_{1}^{2};\gamma;\chi_{4}).
\end{multline}
If $n$ is odd, then
\begin{multline}\label{eq:mndodd}
M^{ND}=\frac{8\zeta(2s)}{\pi^{s-1/2}}\sum_{l=1}^{\infty}\frac{1}{l^s}\sum_{\gamma\in C_{\infty, 0}(64)}\frac{1}{\gamma}\psi\left(\frac{4\pi ln}{\gamma} \right)S_{\infty 0}(l^2,n^2;\gamma;\chi_{4})\\
+(1-2^{-s})\frac{4\zeta(2s)}{\pi^{s-1/2}}\sum_{l=1}^{\infty}\frac{1}{l^s}\sum_{\gamma\in C_{\infty, 0}(16)}\frac{1}{\gamma}\psi\left(\frac{4\pi ln}{\gamma} \right)S_{\infty 0}(l^2,n^2;\gamma;\chi_{4}).
\end{multline}
\end{lem}
\begin{proof}
First, assume that $n$ is even. Then applying Lemmas \ref{lem:supqisodd}, \ref{lem:supqiseven}, \ref{lem:qnleven} we infer
\begin{multline}\label{ND:even}
M^{ND}=
\frac{\zeta(2s)}{\pi^{s-1/2}}\sum_{l\equiv 0\pmod{2}}\frac{1}{l^s}\sum_{q\equiv 0\pmod{4}}\frac{2i}{q}f\left(\frac{4\pi nl}{q} \right)\\ \times \sum_{\substack{a,b\pmod{q}\\ab\equiv 1\pmod{q}}}\chi_4(a)e\left(-\frac{a(l/2)^2+b(n/2)^2}{q} \right)+\\
\frac{\zeta(2s)}{\pi^{s-1/2}}\sum_{l=1}^{\infty}\frac{1}{l^s}\sum_{q=1}^{\infty}\frac{\chi_4(q)}{q}f\left(\frac{4\pi nl }{q} \right)\sum_{\substack{a,b\pmod{q}\\ab\equiv 1\pmod{q}}}e\left( \frac{\overline{4}_q bn^2+\overline{4}_q al^2}{q}\right).
\end{multline}

Since $\chi_4(a)=\chi_4(b)$ for $q\equiv 0\pmod{4}$, the first summand in \eqref{ND:even} is equal to
\begin{equation*}
-\frac{2i\zeta(2s)}{\pi^{s-1/2}}\sum_{l\equiv 0\pmod{2}}\frac{1}{l^s}\sum_{\gamma \in C_{\infty, \infty}(4)}\frac{1}{\gamma}\psi\left(\frac{\pi ln}{\gamma} \right)S_{\infty\infty}\left(\frac{l^2}{4},\frac{n^2}{4};\gamma;\chi_4\right).
\end{equation*}

Using \eqref{eq:sinf0}, we obtain
\begin{equation}
\chi_4(q)\sum_{\substack{a,b\pmod{q}\\ab\equiv 1\pmod{q}}}e\left( \frac{\overline{4}_q bn^2+\overline{4}_q al^2}{q}\right)=S_{\infty 0}\left(l^2,n_{1}^{2};q\sqrt{4};\chi_4\right).
\end{equation}

Therefore, the second summand in \eqref{ND:even} is equal to
\begin{equation*}
\frac{2\zeta(2s)}{\pi^{s-1/2}}\sum_{l=1}^{\infty}\frac{1}{l^s}\sum_{\gamma\in C_{\infty, 0}(4)}\frac{1}{\gamma}\psi\left(\frac{4\pi ln_1}{\gamma} \right)S_{\infty 0}(l^2,n_{1}^{2};\gamma;\chi_{4}).
\end{equation*}

This completes the proof of \eqref{eq:mndeven}.

Now we assume that $n$ is odd. Using Lemmas \ref{lem:supqisodd}, \ref{lem:supqiseven}, \ref{lem:qevennlodd} we obtain
\begin{multline}\label{mnd:odd}
M^{ND}=\frac{\zeta(2s)}{\pi^{s-1/2}}\sum_{(l,2)=1}\frac{1}{l^s}\sum_{q\equiv 2\pmod{4}}\frac{2\chi_4(q/2)}{q}f\left(\frac{4\pi nl}{q} \right)\\
 \times \sum_{\substack{a,b\pmod{r}\\ab\equiv 1\pmod{r}}}e\left(\frac{a\overline{8}_r n^2+b\overline{8}_r l^2}{r} \right)\\
+\frac{\zeta(2s)}{\pi^{s-1/2}}\sum_{l=1}^{\infty}\frac{1}{l^s}\sum_{q=1}^{\infty}\frac{\chi_4(q)}{q}f\left(\frac{4\pi nl}{q} \right)\sum_{\substack{a,b\pmod{q}\\ab\equiv 1\pmod{q}}}e\left( \frac{\overline{4}_q bn^2+\overline{4}_q al^2}{q}\right),
\end{multline}
where $r=q/2$.
Note that $\chi_4$ can be extended to $\chi_{16}$ as follows
\begin{equation}
\chi_{16}(q)=\begin{cases}
\chi_4(q)\quad \text{if }(q,16)=1\\
0\quad \text{otherwise}.
\end{cases}
\end{equation}
Consequently, $\chi_4(q)=\chi_{16}(q)$ for all $(q,2)=1$.
Therefore,  the second summand in \eqref{mnd:odd} is equal to
\begin{equation}
\frac{4\zeta(2s)}{\pi^{s-1/2}}\sum_{l=1}^{\infty}\frac{1}{l^s}\sum_{\gamma\in C_{\infty, 0}(16)}\frac{1}{\gamma}\psi\left(\frac{4\pi ln}{\gamma} \right)S_{\infty 0}(l^2,n^2;\gamma,\chi_{4}).
\end{equation}

In order to evaluate the first summand in \eqref{mnd:odd}, we split the sum over $l$ as 
$$\sum_{(l,2)=1}=\sum_{l=1}^{\infty}-\sum_{l\equiv 0\pmod{2}},$$
which yields
\begin{multline}
\frac{8\zeta(2s)}{\pi^{s-1/2}}\sum_{l=1}^{\infty}\frac{1}{l^s}\sum_{\gamma\in C_{\infty,0}(64)}\frac{1}{\gamma}\psi\left(\frac{4\pi ln}{\gamma} \right)S_{\infty 0}(l^2,n^2;\gamma,\chi_{4})\\-
\frac{4\zeta(2s)}{2^{s}\pi^{s-1/2}}\sum_{l=1}^{\infty}\frac{1}{l^s}\sum_{\gamma\in C_{\infty, 0}(16)}\frac{1}{\gamma}\psi\left(\frac{4\pi ln}{\gamma} \right)S_{\infty 0}(l^2,n^2;\gamma,\chi_{4}).
\end{multline}

This completes the proof.
\end{proof}

\section{Special functions}\label{sec:special}
In this section, we study the Bessel transforms $\psi_H(k)$ and $\psi_D(t)$ defined by \eqref{eq:psiHk} and  \eqref{eq:psiDt} respectively. 
Our goal is to prove integral representations for these functions in terms of the Gauss hypergeometric function.

First, recall that in our case $\kappa=1$ and $k\equiv \kappa \pmod{2}$. Therefore, we can write $k=2m+1$ for $m \in \N$ and
\begin{equation}\label{eq:psiHk2}
\psi_H(k)=4i^{2m+1}\int_{0}^{\infty}J_{2m}(x)\psi(x)\frac{dx}{x},
\end{equation}
\begin{equation}\label{eq:psiDt2}
\psi_D(t)=\frac{2\pi i t}{\sinh{(\pi t)}}\int_{0}^{\infty}(J_{2it}(x)+J_{-2it}(x))\psi(x)\frac{dx}{x},
\end{equation}
where for $a<0$ (see \eqref{eq:fwsx})
\begin{equation}
\psi(x)=f(\omega,s;4x)=\frac{1}{2\pi i}\int_{(a)}\frac{\Gamma(1/2-\alpha/2)}{\Gamma(\alpha/2)}\widehat{\omega}(\alpha)\left( \frac{x}{n}\right)^{\alpha+s-1}d\alpha.
\end{equation}
The function $\psi(x)$ has another integral representation: 
\begin{equation}\label{eq:anotherreprpsi}
\psi(x)=\frac{2}{\sqrt{\pi}}\left(\frac{x}{n}\right)^s\int_{0}^{\infty}\omega(y)\cos\left( \frac{2xy}{n}\right)dy,
\end{equation}
see \cite[Eq. (1.7)]{BF2} and \cite[p. 1984]{BF2} for the proof.

\begin{lem}\label{lem:psiH} The following identity holds
\begin{multline}\label{eq:inttransformhyp}
\psi_H(k)= \frac{4i}{\pi }\Gamma\left(\frac{s}{2}+m\right)\Biggl[\frac{\Gamma\left(s/2-m\right)}{\Gamma(1/2)}\sin{\left( \frac{\pi s}{2}\right)}
\frac{2^{s}}{n^s } \times\\
\int_{0}^{n/2}\omega(y)F\left(\frac{s}{2}+m,\frac{s}{2}-m;\frac{1}{2};\left(\frac{2y}{n}\right)^2\right)dy
+ \frac{\Gamma\left(m+s/2+1/2\right)}{\Gamma(2m+1)}\times \\ \cos{\left( \frac{\pi s}{2}\right)} \frac{n^{2m}}{2^{2m}}
\int_{n/2}^{\infty}\omega(y) F\left(m+\frac{s}{2},m+\frac{s+1}{2};2m+1;\left(\frac{n}{2y}\right)^2\right)dy\Biggr].
\end{multline}

\end{lem}
\begin{proof}
Substituting \eqref{eq:anotherreprpsi} into \eqref{eq:psiHk2} we have
\begin{equation}
\psi_H(k)=\frac{8i^{2m+1}}{\sqrt{\pi}n^s}\int_{0}^{\infty}\omega(y)\int_{0}^{\infty}J_{2m}(x)x^{s-1}\cos \left(x\frac{2y}{n} \right)dxdy.
\end{equation}
The outer integral over $y$ can be split in two parts: $\int_{0}^{\infty}=\int_{0}^{n/2}+\int_{n/2}^{\infty}.$
For each of these parts we evaluate the inner integral over $x$. When $2y/n<1$ we apply \cite[6.699 (2)]{GR} and Euler's reflection formula, so that
\begin{multline}
\int_{0}^{\infty}J_{2m}(x)x^{s-1}\cos \left(x\frac{2y}{n} \right)dx=\frac{2^{s-1}(-1)^m\sin(\pi s/2)}{\pi}\\ \times \Gamma(s/2+m)\Gamma(s/2-m)
F\left(\frac{s}{2}+m,\frac{s}{2}-m;\frac{1}{2};\left(\frac{2y}{n}\right)^2\right).
\end{multline}
This gives the first summand in \eqref{eq:inttransformhyp}. The second summand is obtained similarly using \cite[6.699 (2)]{GR} and the fact that $2y/n\geq 1$.
\end{proof}

\begin{lem}\label{lem:psiD} 
The following identity holds
\begin{equation}\label{eq:psiD}
\psi_D(t)=\frac{2\pi i t}{\sinh{(\pi t)}}\left(h_1(t)+h_1(-t)+h_2(t)\right),
\end{equation}
where
\begin{multline}\label{eq:h1t}
h_1(t)=\frac{\cos{(\pi(s/2+it))}}{\pi}\frac{\Gamma(s/2+it)\Gamma(s/2+1/2+it)}{\Gamma(1+2it)}\\
\times \int_{n/2}^{\infty}\frac{\omega(y)}{y^s}\left(\frac{2y}{n} \right)^{-2it}
F\left(\frac{s}{2}+it,\frac{s+1}{2}+it;1+2it;\left(\frac{n}{2y}\right)^2\right)dy,
\end{multline}
\begin{multline}\label{eq:h1t}
h_2(t)=\frac{2\cosh{(\pi t)}}{\pi}\frac{\sin{(\pi s/2)}}{(n/2)^s}\frac{\Gamma(s/2+it)\Gamma(s/2-it)}{\Gamma(1/2)}\\ 
\times \int_{0}^{n/2}\omega(y)F\left(\frac{s}{2}+it,\frac{s}{2}-it;\frac{1}{2};\left(\frac{2y}{n} \right)^2 \right)dy.
\end{multline}
\end{lem}

\begin{proof}
Substituting \eqref{eq:anotherreprpsi} into \eqref{eq:psiDt2} we show that 
\begin{equation}
\psi_D(t)=\frac{2\pi i t}{\sinh{(\pi t)}}\left(h_1(t)+h_1(-t)+h_2(t)\right),
\end{equation}
where
\begin{equation*}
h_1(t)=\frac{2}{\sqrt{\pi} }\frac{1}{n^s}\int_{n/2}^{\infty}\omega(y)\int_{0}^{\infty}J_{2it}(x)x^{s-1}\cos\left(\frac{2xy}{n} \right)dxdy,
\end{equation*}
\begin{equation*}
h_2(t)=\frac{2 }{\sqrt{\pi}}\frac{1}{n^s}\int_{0}^{n/2}\omega(y)\int_{0}^{\infty}(J_{2it}(x)+J_{-2it}(x))x^{s-1}\cos\left(\frac{2xy}{n} \right)dxdy.
\end{equation*}
Evaluating the inner integrals with respect to $x$ in $h_1(t)$ and $h_2(t)$ using \cite[Eq. 6.699 (2)]{GR}, we complete the proof.

\end{proof}
\begin{lem}\label{lem:h1h2h3}
We have
\begin{equation}\label{eq:h11}
h_1\left(\frac{1-s}{2i}\right)=0,
\end{equation}
\begin{equation}\label{eq:h12}
h_1\left(\frac{s-1}{2i}\right)=\frac{\Gamma(s-1/2)\sin(\pi s)}{\pi (n/2)^{1-s}}\int_{n/2}^{\infty}\omega(y)(y^2-n^2/4)^{1/2-s}dy,
\end{equation}
\begin{multline}\label{eq:h21}
h_2\left(\frac{1-s}{2i}\right)=h_2\left(\frac{s-1}{2i}\right)\\=\frac{2}{\pi}\sin^2(\pi s/2)\frac{\Gamma(s-1/2)}{(n/2)^{1-s}}\int_{0}^{n/2}\omega(y)(n^2/4-y^2)^{1/2-s}dy.
\end{multline}

\end{lem}
\begin{proof}
The identity \eqref{eq:h11} follows directly from \eqref{eq:h1t}. In order to prove \eqref{eq:h12} and \eqref{eq:h21}, we apply \cite[Eq. 15.4.6]{HMF}. Accordingly,
\begin{equation}
F\left( s-\frac{1}{2},s,s;\left(\frac{n}{2y} \right)^2\right)=\frac{(y^2-n^2/4)^{1/2-s}}{y^{1-2s}},
\end{equation}
\begin{equation}
F\left( \frac{1}{2},s-\frac{1}{2},\frac{1}{2};\left(\frac{2y}{n} \right)^2\right)=\frac{(n^2/4-y^2)^{1/2-s}}{(n/2)^{1-2s}}.
\end{equation}
The statement follows. 
\end{proof}

\begin{cor}
The following identity holds
\begin{multline}\label{eq:psiD1-s}
\frac{\psi_D\left(\frac{1-s}{2i}\right)\sinh \left(\pi \frac{1-s}{2i}\right)}{(1-s)\cosh\left(\pi\frac{1-s}{2i}\right)}=\\
\frac{2\Gamma(s-1/2)}{(n/2)^{1-s}} \biggl(
\sin (\pi s/2)\int_{0}^{n/2}\omega(y)\left( \frac{n^2}{4}-y^2\right)^{1/2-s}dy\\
+\cos(\pi s/2)\int_{n/2}^{\infty}\omega(y)\left(y^2- \frac{n^2}{4}\right)^{1/2-s}dy\biggr).
\end{multline}
\end{cor}
\begin{proof}
This is a consequence of Lemma \ref{lem:h1h2h3} and \eqref{eq:psiD}.
\end{proof}

\section{Fourier coefficients of Eisenstein series}\label{sec:cont1}

As a preliminary step towards understanding the continuous spectrum arising from the Kuznetsov trace formula, we compute explicitly some Fourier coefficients of Eisenstein series for the groups $\Gamma_0(4)$, $\Gamma_0(16)$ and $\Gamma_0(64)$. To this end, it is required to determine a list of singular cusps for the considered groups and to compute various characters appearing as a part of Fourier coefficients. 

\subsection{Singular cusps for $\Gamma_0(4)$, $\Gamma_0(16)$ and $\Gamma_0(64)$}
Let $\sigma_{\mathfrak{a}}$ be a scaling matrix for  a cusp $\mathfrak{a}$ and
let $\lambda_{\mathfrak{a}}$ be defined by $\sigma_{\mathfrak{a}}^{-1}\lambda_{\mathfrak{a}}\sigma_{\mathfrak{a}}=\begin{pmatrix}1&1\\0&1\end{pmatrix}$.
Recall that $\mathfrak{a}$ is called singular for $\chi$ if $\chi(\lambda_{\mathfrak{a}})=1$. It follows from \cite[Proposition 3.3]{KY} that  if $\mathfrak{c}=1/\omega$ is a cusp of $\Gamma=\Gamma_0(N)$ and
\begin{equation}
N=(N,\omega)N', \quad \omega=(N,\omega)\omega', \quad N'=(N',\omega)N'',
\end{equation}
then we have 
\begin{equation}\label{eq:singcusp}
\lambda_{1/\omega}=\begin{pmatrix}
1-\omega N''& N''\\
-\omega^2N''&1+\omega N''
\end{pmatrix}.
\end{equation}

\begin{lem}\label{lem:singularcusps}
The following cusps are singular for $\Gamma_0(4)$: $$0, \infty.$$

The following cusps are singular for $\Gamma_0(16)$: $$0, 1/2, 1/4, 1/8, 1/12,  \infty.$$

The following cusps are singular for $\Gamma_0(64)$: $$ 0, 1/2, 1/4, 1/8, 1/12, 1/16, 1/24, 1/32, 1/40, 1/48, 1/56,  \infty.$$

\end{lem}
\begin{proof}
According to \cite[Corollary 3.2]{KY}, a complete set of representatives for the set of inequivalent cusps of  $\Gamma=\Gamma_0(N)$ is given by $\frac{1}{w}=\frac{1}{uf}$, where  $f$ runs over divisors of $N$ and $u$ runs modulo $(f,N/f)$ such that $u$ is coprime to $(f,N/f)$, where $u$ is chosen so that $(u,N)=1$ after adding a suitable multiple of $(f,N/f)$. 

Consequently, the group $\Gamma_0(4)$ has the following inequivalent cusps: $$\frac{1}{1}\sim 0,\frac{ 1}{2},\frac{1}{4}\sim \infty.$$ For the group $\Gamma_0(16)$ there are six inequivalent cusps: $$\frac{1}{1}\sim 0, \frac{1}{2}, \frac{1}{4}, \frac{1}{8}, \frac{1}{12}, \frac{1}{16} \sim \infty,$$ and for
the group $\Gamma_0(64)$ twelve: $$ \frac{1}{1}\sim 0, \frac{1}{2}, \frac{1}{4}, \frac{1}{8}, \frac{1}{12}, \frac{1}{16}, \frac{1}{24}, \frac{1}{32}, \frac{1}{40}, \frac{1}{48}, \frac{1}{56}, \frac{1}{64}\sim \infty.$$ 

In order to check which of these cusps are singular we use \eqref{eq:singcusp}. Accordingly, the cusp is singular if 
\begin{equation}
\chi_4(1+wN'')=1.
\end{equation}
Verifying this condition, we find that for $\Gamma_0(16)$ and $\Gamma_0(64)$ all the cusps listed above are singular. For $\Gamma_0(4)$ the cusps $0, \infty$ are singular, but $1/2$ is not because in this case $\chi_4(1+wN'')=\chi_4(3)=-1$.
\end{proof}

\subsection{Computations with characters}

\begin{lem}
Let $\mathfrak{c}=1/w$ be any cusp and $\mathfrak{a}=1/r$ be an Atkin-Lehner cusp of $\Gamma_0(N)$. Let us choose the scaling matrices as follows:
\begin{equation}
\sigma_{\mathfrak{c}}=\begin{pmatrix}
1&0\\
w&1
\end{pmatrix}
\begin{pmatrix}
\sqrt{N''}&0\\
0&1/\sqrt{N''}
\end{pmatrix},
\end{equation}

\begin{equation}
\sigma_{\mathfrak{a}}=\begin{pmatrix}
1&(\overline{s}s-1)/r\\
r&\overline{s}s
\end{pmatrix}
\begin{pmatrix}
\sqrt{s}&0\\
0&1/\sqrt{s}
\end{pmatrix}.
\end{equation}
Then
\begin{equation}\label{eq:coset}
\sigma_{\mathfrak{c}}^{-1}\Gamma\sigma_{\mathfrak{a}}=\left\{
\begin{pmatrix}
\frac{A}{N''}\sqrt{N''s}& \frac{B}{N''}\sqrt{\frac{N''}{s}}\\
C\sqrt{N''s}&D\sqrt{\frac{N''}{s}}
\end{pmatrix}:
\begin{aligned}
 C\equiv -wA\pmod{r}\\ D\equiv -wB\pmod{s}\\AD-BC=1
\end{aligned}
\right\}
\end{equation}
and for $\rho \in \sigma_{\mathfrak{c}}^{-1}\Gamma\sigma_{\mathfrak{a}}$ we have
\begin{equation}\label{eq:charcoset}
\chi_4(\sigma_{\mathfrak{c}}\rho\sigma_{\mathfrak{a}}^{-1})=\chi_4\left((wA+C)\frac{1-s\overline{s}}{r}+wB+D \right).
\end{equation}
In particular,
\begin{equation}\label{eq:charcoset1}
\chi_4(\sigma_{\mathfrak{c}}\rho\sigma_{\infty}^{-1})=\chi_4\left(wB+D \right),
\end{equation}
\begin{equation}\label{eq:charcoset2}
\chi_4(\sigma_{\mathfrak{c}}\rho\sigma_{0}^{-1})=\chi_4\left(wA+C \right).
\end{equation}
\end{lem}
\begin{proof}
For the proof of \eqref{eq:coset} see \cite[Lemma 3.5]{KY}. By direct computations we find that
\begin{equation}
\sigma_{\mathfrak{c}}\rho\sigma_{\mathfrak{a}}^{-1}=\begin{pmatrix}
*&*\\
*&(wA+C)\frac{1-\overline{s}s}{r}+WB+D
\end{pmatrix}.
\end{equation}
This implies \eqref{eq:charcoset}. If $\mathfrak{a}=\infty$, this formula can be simplified further by noting that in this case $r=N$, $s=1$, and therefore,
$$\frac{1-\overline{s}s}{r}=0.$$
Similarly, for $\mathfrak{a}=\infty$, we have $r=1$, $s=N$ and
$$\frac{1-\overline{s}s}{r}=1-N\overline{N}=1-N.$$ Then using the relation $wB+D\equiv 0\pmod{N}$ in \eqref{eq:coset} we prove \eqref{eq:charcoset2}.
\end{proof}

\begin{cor}\label{cor:gamma04}
For the group $\Gamma_0(4)$ we have
\begin{equation}
\chi_4(\sigma_{\infty}\rho\sigma_{\infty}^{-1})=\chi_4(D),\quad \chi_4(\sigma_{0}\rho\sigma_{\infty}^{-1})=\chi_4(-C),
\end{equation}
\begin{equation}
\chi_4(\sigma_{\infty}\rho\sigma_{0}^{-1})=\chi_4(C),\quad \chi_4(\sigma_{0}\rho\sigma_{0}^{-1})=\chi_4(D).
\end{equation}
\end{cor}
\begin{proof}
If $\mathfrak{a}=\mathfrak{c}=\infty$, then $w=N=4$ and by \eqref{eq:charcoset1} we have $\chi_4(\sigma_{\mathfrak{c}}\rho\sigma_{\mathfrak{a}}^{-1})=\chi_4(D)$.

If $\mathfrak{a}=\infty$, $\mathfrak{c}=0$, then $w=1$ and by \eqref{eq:charcoset1} we have $\chi_4(\sigma_{\mathfrak{c}}\rho\sigma_{\mathfrak{a}}^{-1})=\chi_4(B+D)$.
According to \eqref{eq:coset}
\begin{equation}
\begin{cases}
AD-BC=1\\
C\equiv-A\pmod{N}
\end{cases}.
\end{equation}
Therefore,
$A(B+D)\equiv 1\pmod{N}$ and $B+D\equiv \overline{A}\pmod{N}\equiv - \overline{C}\pmod{N}$. Consequently,  $\chi_4(\sigma_{\mathfrak{c}}\rho\sigma_{\mathfrak{a}}^{-1})=\chi_4(- \overline{C})=\chi_4(- C).$

If $\mathfrak{a}=0$, $\mathfrak{c}=\infty$, then $w=N=4$ and by \eqref{eq:charcoset2} we have $\chi_4(\sigma_{\mathfrak{c}}\rho\sigma_{\mathfrak{a}}^{-1})=\chi_4(C)$.

If $\mathfrak{a}=\mathfrak{c}=0$, then $w=1$ and by \eqref{eq:charcoset2} we have $\chi_4(\sigma_{\mathfrak{c}}\rho\sigma_{\mathfrak{a}}^{-1})=\chi_4(A+C)$.
Using \eqref{eq:coset} we find that
\begin{equation}
\begin{cases}
AD-BC=1\\
D\equiv-B\pmod{N}
\end{cases}.
\end{equation}
Therefore,
$D(A+C)\equiv 1\pmod{N}$ and $\chi_4(\sigma_{\mathfrak{c}}\rho\sigma_{\mathfrak{a}}^{-1})=\chi_4(\overline{D})=\chi_4(D).$

\end{proof}

\begin{cor}\label{cor:characters}
For the groups $\Gamma_0(16)$ and $\Gamma_0(64)$ we have
\begin{equation}
\chi_4(\sigma_{0}\rho\sigma_{0}^{-1})=\chi_4(D),\quad \chi_4(\sigma_{0}\rho\sigma_{\infty}^{-1})=\chi_4(-C),
\end{equation}
\begin{equation}
\chi_4(\sigma_{1/2}\rho\sigma_{\infty}^{-1})=\chi_4(-C/2),\quad \chi_4(\sigma_{1/2}\rho\sigma_{0}^{-1})=\chi_4(D/2).
\end{equation}
For all other singular cusps $\mathfrak{c}$ of  $\Gamma_0(16)$ and $\Gamma_0(64)$ the following holds
\begin{equation}
\chi_4(\sigma_{\mathfrak{c}}\rho\sigma_{\infty}^{-1})=\chi_4(D),\quad \chi_4(\sigma_{\mathfrak{c}}\rho\sigma_{0}^{-1})=\chi_4(C).
\end{equation}

\end{cor}
\begin{proof}
Note that if $w\equiv 0\pmod{4}$, then
 \begin{equation}
\chi_4(\sigma_{\mathfrak{c}}\rho\sigma_{\infty}^{-1})=\chi_4(D),\quad \chi_4(\sigma_{\mathfrak{c}}\rho\sigma_{0}^{-1})=\chi_4(C).
\end{equation}
This is the case for all singular cusps of $\Gamma_0(16)$ and $\Gamma_0(64)$ except $\mathfrak{c}=1/1\sim 0$ and $\mathfrak{c}=1/2$.

For $\mathfrak{c}=1/1\sim 0$ all computations are exactly the same as in Corollary \ref{cor:gamma04}.

Consider $\mathfrak{c}=1/2$. It follows from \eqref{eq:charcoset1} that \begin{equation}
\chi_4(\sigma_{1/2}\rho\sigma_{\infty}^{-1})=\chi_4(2B+D).
\end{equation}
According to \eqref{eq:coset}
\begin{equation}
\begin{cases}
AD-BC=1\\
C\equiv -2A\pmod{N},
\end{cases}
\end{equation}
which implies that $A(D+2B)\equiv 1\pmod{N}$ and $\chi_4(\sigma_{1/2}\rho\sigma_{\infty}^{-1})=\chi_4(\overline{A}).$
Since $A\overline{A}\equiv 1\pmod{4}$ and $A\equiv -C/2\pmod{N/2}$ we have 
\begin{equation}
\chi_4(\sigma_{1/2}\rho\sigma_{\infty}^{-1})=\chi_4(\overline{A})=\chi_4(A)=\chi_4(-C/2).
\end{equation}
As a consequence of \eqref{eq:charcoset2} we find that
\begin{equation}
 \chi_4(\sigma_{1/2}\rho\sigma_{0}^{-1})=\chi_4(2A+C).
\end{equation}
According to \eqref{eq:coset}
\begin{equation}
\begin{cases}
AD-BC=1\\
D\equiv -2B\pmod{N}.
\end{cases}
\end{equation}
Therefore, $-B(2A+C)\equiv 1\pmod{N}$ and
\begin{equation}
\chi_4(\sigma_{1/2}\rho\sigma_{0}^{-1})=\chi_4(-\overline{B})=\chi_4(-B)=\chi_4(D/2).
\end{equation}

\end{proof}

\subsection{Fourier coefficients}

In this section we assume that $m$ is positive.
For the sake of brevity, we introduce the notation:
\begin{equation}
\delta_{n}(m):=\begin{cases}
1 \quad\text{if }n|m\\
0\quad \text{otherwise}
\end{cases},
\end{equation}
\begin{equation}\label{eq:sm}
s(m):=\frac{\sigma_{1-2s}(\chi_4;m)}{L(\chi_4,2s)},
\end{equation}
\begin{equation}\label{eq:tm}
t(m):=\frac{\tau(\chi_4)\sigma_{2s-1}(\chi_4;m)m^{1-2s}}{L(\chi_4,2s)},
\end{equation}
\begin{equation}
\sum_{a\pmod{m}}^{*}:=\sum_{\substack{a\pmod{m}\\(a,m)=1}}.
\end{equation}

In this section, we will frequently use the following lemma, which follows directly from the Chinese remainder theorem.
\begin{lem}\label{lem:Chineseremainder}
For $(m,n)=1$ we have
\begin{equation*}
\sum_{c\pmod{mn}}^{*}f(c)=\sum_{a\pmod{m}}^{*}\sum_{b\pmod{n}}^{*}f(an\overline{n}_m+bm\overline{m}_n),
\end{equation*}
where $n\overline{n}_m\equiv 1\pmod{m}$ and $m\overline{m}_n\equiv 1\pmod{n}$.
\end{lem}

Now we are ready to compute the Fourier coefficients \eqref{eq:coeffFourierEisent} for $N=4,16,64$ and $\mathfrak{a}=0, \infty$. 
We provide a complete proof for the case $N=64$ and $\mathfrak{a}=\infty$.  The other cases can be proved by the same manner.

\begin{lem}\label{lem:64inf}
Let $N=64$ and $\mathfrak{a}=\infty$. Then
\begin{equation}
\phi_{\infty,0}(m,s,\chi_4)=\chi_4(-1)\frac{s(m)}{8^{2s}},
\end{equation}
\begin{equation}
\phi_{\infty,1/2}(m,s,\chi_4)=\chi_4(-1)e\left(\frac{m}{2}\right)\frac{s(m)}{8^{2s}},
\end{equation}
\begin{equation}
\phi_{\infty,\frac{1}{4u}}(m,s,\chi_4)=\chi_4(-u)e\left(-\frac{mu}{4}\right)\frac{s(m)}{8^{2s}},\quad u=1,3,
\end{equation}
\begin{equation}\label{eq:64inf1/8u}
\phi_{\infty,\frac{1}{8u}}(m,s,\chi_4)=\chi_4(-u)e\left(-\frac{mu}{8}\right)\frac{s(m)}{8^{2s}},\quad u=1,3,5,7,
\end{equation}
\begin{equation}\label{eq:64inf1/16u}
\phi_{\infty,\frac{1}{16u}}(m,s,\chi_4)=\chi_4(-u)4\delta_{4}(m)e\left(-\frac{mu}{16}\right)\frac{s(m)}{16^{2s}},\quad u=1,3,
\end{equation}
\begin{equation}\label{eq:64infphi1/32}
\phi_{\infty,\frac{1}{32}}(m,s,\chi_4)=\frac{8}{(32)^{2s}}\delta_{8}(m)t\left(\frac{m}{8}\right)-\frac{16}{(64)^{2s}}\delta_{16}(m)t\left(\frac{m}{16}\right),
\end{equation}
\begin{equation}\label{eq:64infphi1/64}
\phi_{\infty,\infty}(m,s,\chi_4)=\frac{16}{(64)^{2s}}\delta_{16}(m)t\left(\frac{m}{16}\right).
\end{equation}
\end{lem}

\begin{proof}
We need to evaluate \eqref{eq:coeffFourierEisent} for  $\mathfrak{a}=\infty$.  Consequently, \cite[Eq. 3.20]{KY} can be simplified as follows for $r=N=64$:
\begin{multline}
\Gamma_{\infty}\setminus\sigma_{\mathfrak{c}}^{-1}\Gamma\sigma_{\mathfrak{\infty}}/\Gamma_{\infty}=\\
\left\{
\begin{pmatrix}
*& *\\
C\sqrt{N''}&D\sqrt{N''}
\end{pmatrix}:
\begin{aligned}
 D\pmod{C}, \quad (D,C)=1, \quad (C,N)=f \\ D\equiv -\overline{(C/f)}u\pmod{(f,N/f)}
\end{aligned}
\right\}.
\end{multline}

If $\mathfrak{c}=0$, then  $\chi_4(\sigma_{0}\rho\sigma_{\infty}^{-1})=\chi_4(-C)$ by Corollary \ref{cor:characters},  and
$$f=1, \quad(f,N/f)=1,  \quad N''=64, \quad D\equiv -\overline{C}\pmod{1}.$$
As a result, we conclude that
\begin{equation}
\phi_{\infty,0}(m,s,\chi_4)=\sum_{(C,2)=1}\frac{\chi_4(-C)}{(8C)^{2s}}\sum_{D\pmod{C}}^{*}e\left( \frac{mD}{C}\right).
\end{equation}
The condition $(C,2)=1$ can be omitted since $\chi_4(-C)=0$ if this doesn't hold.
Furthermore, we use the following identity for the inner sum
\begin{equation}
\sum_{D\pmod{C}}^{*}e\left( \frac{mD}{C}\right)=\sum_{d|(m,C)}d\mu(C/d),
\end{equation}
and interchange the order of summations, getting
\begin{multline}
\phi_{\infty,0}(m,s,\chi_4)
=\frac{\chi_4(-1)}{8^{2s}}\sum_{d|m}d\sum_{C\equiv0\pmod{d}}\frac{\chi_4(C)}{C^{2s}}\mu(C/d)\\
=\frac{\chi_4(-1)}{8^{2s}}\sum_{d|m}\frac{\chi_4(d)}{d^{2s-1}}\sum_{C}\frac{\chi_4(C)}{C^{2s}}\mu(C)=\chi_4(-1)\frac{s(m)}{8^{2s}}.
\end{multline}
If $\mathfrak{c}=\frac{1}{2}$, then  $\chi_4(\sigma_{1/2}\rho\sigma_{\infty}^{-1})=\chi_4(-\overline{C/2})$ by Corollary \ref{cor:characters}, and
$$f=2, \quad(f,N/f)=2,  \quad N''=16, \quad D\equiv -\overline{C/2}\pmod{2}.$$
This yields that
\begin{equation}
\phi_{\infty,\frac{1}{2}}(m,s,\chi_4)=\sum_{(C,2)=1}\frac{\chi_4(-C)}{(8C)^{2s}}\sum_{D\pmod{2C}}^{*}e\left( \frac{mD}{2C}\right).
\end{equation}
By Lemma \ref{lem:Chineseremainder} we have
\begin{multline}
\sum_{D\pmod{2C}}^{*}e\left( \frac{mD}{2C}\right)=\sum_{d_1\pmod{2}}^{*}\sum_{d_2(C)}^{*}e\left(\frac{md_1\overline{C}_2}{2} \right)e\left(\frac{md_2\overline{2}_C}{C} \right)\\
=e\left( \frac{m}{2}\right)\sum_{d_2(C)}^{*}e\left(\frac{md_2\overline{2}_C}{C} \right)=e\left( \frac{m}{2}\right)\sum_{d|(m,C)}d\mu(C/d).
\end{multline}
Consequently,
 \begin{equation}
\phi_{\infty,\frac{1}{2}}(m,s,\chi_4)=\chi_4(-1)e\left(\frac{m}{2}\right)\frac{s(m)}{8^{2s}}.
\end{equation}

If $\mathfrak{c}=\frac{1}{4u}$ with $ u=1,3$, then  $\chi_4(\sigma_{0}\rho\sigma_{\infty}^{-1})=\chi_4(D)$ by Corollary \ref{cor:characters}, and
$$f=4, \quad(f,N/f)=4, \quad N''=4, \quad D\equiv -u\overline{C/4}\pmod{4}.$$
Therefore,
\begin{equation}
\phi_{\infty,\frac{1}{4u}}(m,s,\chi_4)=\sum_{(C,2)=1}\frac{\chi_4(-u\overline{C})}{(8C)^{2s}}\sum_{\substack{D\pmod{4C}\\D\equiv -u\overline{C}\pmod{4}}}^{*}e\left( \frac{mD}{4C}\right).
\end{equation}
Applying Lemma \ref{lem:Chineseremainder}, we infer
\begin{multline}
\sum_{\substack{D\pmod{4C}\\D\equiv -u\overline{C}\pmod{4}}}^{*}e\left( \frac{mD}{4C}\right)=
\sum_{\substack{d_1\pmod{4}\\ d_1\overline{C}_4\equiv -u\pmod{4} }}e\left(\frac{md_1\overline{C}_4}{4} \right)\\ \times
\sum_{d_2\pmod{C}}e\left(\frac{md_2}{C} \right)=e\left( -\frac{mu}{4}\right)\sum_{d|(m,C)}d\mu(C/d).
\end{multline}
This implies that for $ u=1,3$ we have
\begin{equation}
\phi_{\infty,\frac{1}{4u}}(m,s,\chi_4)=\chi_4(-u)e\left(-\frac{mu}{4}\right)\frac{s(m)}{8^{2s}}.
\end{equation}


If $\mathfrak{c}=\frac{1}{8u}$ with $ u=1,3,5,7$, then  $\chi_4(\sigma_{0}\rho\sigma_{\infty}^{-1})=\chi_4(D)$  by Corollary \ref{cor:characters},  and
$$f=8, \quad(f,N/f)=8,  \quad N''=1, \quad D\equiv -u\overline{C/8}\pmod{8}.$$
Similarly to the previous cases
\begin{multline}
\phi_{\infty,\frac{1}{8u}}(m,s,\chi_4)=\sum_{(C,8)=1}\frac{\chi_4(-u\overline{C})}{(8C)^{2s}}\sum_{\substack{D\pmod{8C}\\D\equiv -u\overline{C}\pmod{8}}}^{*}e\left( \frac{mD}{8C}\right)
\\= \chi_4(-u)e\left( -\frac{mu}{8}\right)\sum_{C=1}^{\infty}\frac{\chi_4(C)}{(8C)^{2s}}\sum_{d|(m,C)}d\mu\left(\frac{C}{d} \right)\\ =\chi_4(-u)e\left(-\frac{mu}{8}\right)\frac{s(m)}{8^{2s}}.
\end{multline}

If $\mathfrak{c}=\frac{1}{16u}$ with $ u=1,3$, then  $\chi_4(\sigma_{0}\rho\sigma_{\infty}^{-1})=\chi_4(D)$ by Corollary \ref{cor:characters}, and
$$f=16, \quad(f,N/f)=4, \quad N''=1, \quad D\equiv -u\overline{C/16}\pmod{4}.$$
From this we derive that
\begin{equation}
\phi_{\infty,\frac{1}{16u}}(m,s,\chi_4)=\sum_{(C,4)=1}\frac{\chi_4(-u\overline{C})}{(16C)^{2s}}\sum_{\substack{D\pmod{16C}\\D\equiv -u\overline{C}\pmod{4}}}^{*}e\left( \frac{mD}{16C}\right).
\end{equation}
The inner sum can be evaluated by applying Lemma \ref{lem:Chineseremainder} as follows
\begin{multline}
\sum_{\substack{D\pmod{16C}\\D\equiv -u\overline{C}\pmod{4}}}^{*}e\left( \frac{mD}{16C}\right)\\=
\sum_{\substack{d_1\pmod{16}\\d_1\equiv-u\overline{C}_4\pmod{4}}}^{*}\sum_{d_2\pmod{C}}^{*}
e\left(\frac{md_1\overline{C}_{16}}{16} \right)e\left(\frac{md_2\overline{16}_{C}}{C} \right)\\=
\sum_{\substack{d_1\pmod{16}\\Cd_1\equiv-u\overline{C}_4\pmod{4}}}^{*}
e\left(\frac{md_1}{16} \right)\sum_{d_2\pmod{C}}^{*}e\left(\frac{md_2}{C} \right).
\end{multline}
Note that $\overline{C}_4\equiv C\pmod{4}$, and therefore, the condition $Cd_1\equiv-u\overline{C}_4\pmod{4}$ can be replaced by $d_1\equiv -u\pmod{4}$.

Making the change of variables $d_3:=-u+4d_3$, $d_3\pmod{4}$ we infer
\begin{multline}
\sum_{\substack{d_1\pmod{16}\\d_1\equiv -u\pmod{4}}}e\left( \frac{md_1}{16}\right)=\sum_{d_3\pmod{4}}e\left(- \frac{mu}{16}\right)e\left( \frac{md_3}{4}\right)\\=4\delta_{4}(m)e\left(-\frac{mu}{16}\right).
\end{multline}
Consequently,
\begin{equation}
\phi_{\infty,\frac{1}{16u}}(m,s,\chi_4)=\chi_4(-u)4\delta_{4}(m)e\left(-\frac{mu}{16}\right)\frac{s(m)}{16^{2s}}.
\end{equation}

If $\mathfrak{c}=\frac{1}{32}$, then $\chi_4(\sigma_{0}\rho\sigma_{\infty}^{-1})=\chi_4(D)$ by Corollary \ref{cor:characters},  and
$$f=32, \quad(f,N/f)=2, \quad N''=1, \quad D\equiv -u\overline{C/32}\pmod{2}.$$
In these settings
\begin{equation}
\phi_{\infty,\frac{1}{32}}(m,s,\chi_4)=\sum_{(C,2)=1}\frac{1}{(32C)^{2s}}\sum_{D\pmod{32C}}^{*}\chi_4(D)e\left( \frac{mD}{32C}\right),
\end{equation}
where the inner sum is the Gauss sum $g(\chi_4;32C;m)$ defined by \eqref{eq:miyake}. Then it follows from  \eqref{eq:miyake2} that
\begin{multline*}
\phi_{\infty,\frac{1}{32}}(m,s,\chi_4)=\frac{\tau(\chi_4)}{(32)^{2s}}\sum_{d|m}d\chi_4\left(\frac{m}{d}\right)\sum_{\substack{(C,2)=1\\C\equiv0\pmod{d/(8,d)}}}
\frac{\chi_4\left( \frac{8C}{d}\right)\mu\left( \frac{8C}{d}\right)}{C^{2s}}\\
=\frac{\tau(\chi_4)}{(32)^{2s}}\sum_{\substack{d|m\\(d/(8,d),2)=1}}d\chi_4\left(\frac{m}{d}\right)\left(\frac{(8,d)}{d} \right)^{2s}\sum_{(C,2)=1}
\frac{\chi_4\left( \frac{8C}{(8,d)}\right)\mu\left( \frac{8C}{(8,d)}\right)}{C^{2s}}.
\end{multline*}
Since $\chi_4\left( \frac{8C}{(8,d)}\right)=0$ unless $(8,d)=8$, the expression above simplifies to
\begin{multline*}
\phi_{\infty,\frac{1}{32}}(m,s,\chi_4)=\frac{\tau(\chi_4)}{(32)^{2s}}\sum_{\substack{d|m\\d\equiv 0\pmod{8}\\(d/8,2)=1}}d\chi_4\left(\frac{m}{d} \right)\left(\frac{8}{d} \right)^{2s}\sum_{(C,2)=1}\frac{\chi_4(C)\mu(C)}{C^{2s}}\\
=\frac{\tau(\chi_4)}{(32)^{2s}L(\chi_4,2s)}8\delta_8(m)\sum_{\substack{d|\frac{m}{8}\\(d,2)=1}}d^{1-2s}\chi_4\left(\frac{m/8}{d}\right).
\end{multline*}
If $m/8$ is odd then the condition $(d,2)=1$ can be removed, and if $m/8$ is even we have $\chi_4\left( \frac{m/8}{d}\right)=0$ for $d$ odd. Consequently,
\begin{equation}\label{eq:1-delta16}
\phi_{\infty,\frac{1}{32}}(m,s,\chi_4)=\frac{8}{(32)^{2s}}\delta_{8}(m)(1-\delta_{16}(m))t\left(\frac{m}{8}\right).
\end{equation}

For $m \equiv 0 \pmod{16}$ the following identity holds
\begin{equation}
\left( \frac{m}{8}\right)^{1-2s}\sigma_{2s-1}(\chi_4;m/8)=2^{1-2s}\left( \frac{m}{16}\right)^{1-2s}\sigma_{2s-1}(\chi_4;m/16),
\end{equation}
and therefore,
\begin{equation}
\phi_{\infty,\frac{1}{32}}(m,s,\chi_4)=\frac{8}{(32)^{2s}}\delta_{8}(m)t\left(\frac{m}{8}\right)-\frac{16}{(64)^{2s}}\delta_{16}(m)t\left(\frac{m}{16}\right).
\end{equation}

If $\mathfrak{c}=\frac{1}{64}$, then  $\chi_4(\sigma_{0}\rho\sigma_{\infty}^{-1})=\chi_4(D)$ by Corollary \ref{cor:characters}, and
$$f=64, \quad(f,N/f)=1,  \quad N''=1, \quad D\equiv -u\overline{C/64}\pmod{1}.$$
This implies that
\begin{equation}
\phi_{\infty,\infty}(m,s,\chi_4)=\sum_{C}\frac{1}{(64C)^{2s}}\sum_{D\pmod{64C}}^{*}\chi_4(D)e\left( \frac{mD}{64C}\right).
\end{equation}

The inner sum is the Gauss sum $g(\chi_4;64C;m)$ defined by \eqref{eq:miyake}. Applying the representation \eqref{eq:miyake2} for this sum, we have

\begin{multline*}
\phi_{\infty,\infty}(m,s,\chi_4)=\frac{\tau(\chi_4)}{(64)^{2s}}\sum_{C=1}^{\infty}\frac{1}{C^{2s}}\sum_{d|(m,16C)}d\chi_4\left(\frac{16C}{d}\right)\chi_4\left(\frac{n}{d} \right)\mu\left( \frac{16C}{q}\right)\\
=\frac{\tau(\chi_4)}{(64)^{2s}}\sum_{d|m}d\chi_4\left( \frac{m}{d}\right)\left(\frac{(16,d)}{d} \right)^{2s}\sum_{C=1}^{\infty}\frac{\chi_4\left( \frac{16C}{(16,d)}\right)\mu\left( \frac{16C}{(16,d)}\right)}{C^{2s}}.
\end{multline*}
We remark that $\chi_4\left( \frac{16C}{(16,d)}\right)=0$ unless $\frac{(16,d)}{d} =1$. Therefore, we can assume that $d\equiv 0\pmod{16}$ and 
\begin{multline*}
\phi_{\infty,\infty}(m,s,\chi_4)=\frac{\tau(\chi_4)}{(64)^{2s}}\sum_{\substack{d|m\\d\equiv 0\pmod{16}}}d\chi_4\left( \frac{m}{d}\right)\left(\frac{16}{d} \right)^{2s}\sum_{C=1}^{\infty}\frac{\chi_4\left(C\right)\mu\left( C\right)}{C^{2s}}\\
=\frac{\tau(\chi_4)}{(64)^{2s}}\delta_{16}(m)\sum_{d|\frac{m}{16}}\chi_4\left(\frac{m/16}{d} \right)\frac{16d^{-2s+1}}{L(\chi_4,2s)}
=\frac{16}{(64)^{2s}}\delta_{16}(m)t\left(\frac{m}{16}\right).
\end{multline*}


\end{proof}


\begin{lem}\label{lem:640}
Let $N=64$ and $\mathfrak{a}=0$. Then
\begin{equation}
\phi_{0,0}(m,s,\chi_4)=\frac{16}{(64)^{2s}}\delta_{16}(m)t\left(\frac{m}{16}\right),
\end{equation}
\begin{equation}\label{eq:6401/2}
\phi_{0,1/2}(m,s,\chi_4)=\frac{8}{(32)^{2s}}\delta_{8}(m)t\left(\frac{m}{8}\right)-\frac{16}{(64)^{2s}}\delta_{16}(m)t\left(\frac{m}{16}\right),
\end{equation}
\begin{equation}
\phi_{0,\frac{1}{4u}}(m,s,\chi_4)=4\delta_{4}(m)e\left(\frac{mu}{16}\right)\frac{s(m)}{16^{2s}},\quad u=1,3,
\end{equation}
\begin{equation}\label{eq:6401/8u}
\phi_{0,\frac{1}{8u}}(m,s,\chi_4)=e\left(\frac{mu}{8}\right)\frac{s(m)}{8^{2s}},\quad u=1,3,5,7,
\end{equation}
\begin{equation}\label{eq:6401/16u}
\phi_{0,\frac{1}{16u}}(m,s,\chi_4)=e\left(\frac{mu}{4}\right)\frac{s(m)}{8^{2s}},\quad u=1,3,
\end{equation}
\begin{equation}
\phi_{0,\frac{1}{32}}(m,s,\chi_4)=e\left(\frac{m}{2}\right)\frac{s(m)}{8^{2s}},
\end{equation}
\begin{equation}\label{eq:0infphi1/64}
\phi_{0,\infty}(m,s,\chi_4)=\frac{s(m)}{8^{2s}}.
\end{equation}
\end{lem}

\begin{lem}
Let $N=16$ and $\mathfrak{a}=0$. Then
\begin{equation}\label{eq:1600}
\phi_{0,0}(m,s,\chi_4)=\frac{4}{(16)^{2s}}\delta_{4}(m)t\left(\frac{m}{4}\right),
\end{equation}
\begin{equation}\label{eq:1601/2}
\phi_{0,1/2}(m,s,\chi_4)=\frac{2}{8^{2s}}\delta_{2}(m)t\left(\frac{m}{2}\right)-\frac{4}{(16)^{2s}}\delta_{4}(m)t\left(\frac{m}{4}\right),
\end{equation}
\begin{equation}\label{eq:1601/4u}
\phi_{0,\frac{1}{4u}}(m,s,\chi_4)=e\left(\frac{mu}{4}\right)\frac{s(m)}{4^{2s}},\quad u=1,3,
\end{equation}
\begin{equation}\label{eq:1601/8}
\phi_{0,\frac{1}{8}}(m,s,\chi_4)=e\left(\frac{m}{2}\right)\frac{s(m)}{4^{2s}},
\end{equation}
\begin{equation}\label{eq:160inf}
\phi_{0,\infty}(m,s,\chi_4)=\frac{s(m)}{4^{2s}}.
\end{equation}
\end{lem}
\begin{lem}
Let $N=16$ and $\mathfrak{a}=\infty$. Then
\begin{equation}\label{eq:16inf0}
\phi_{\infty,0}(m,s,\chi_4)=\chi_4(-1)\phi_{0,\infty}(m,s,\chi_4)=\chi_4(-1)\frac{s(m)}{4^{2s}},
\end{equation}
\begin{equation}\label{eq:16inf1/2}
\phi_{\infty,1/2}(m,s,\chi_4)=\chi_4(-1)\phi_{0,\frac{1}{8}}(m,s,\chi_4)=\chi_4(-1)e\left(\frac{m}{2}\right)\frac{s(m)}{4^{2s}},
\end{equation}
\begin{multline}\label{eq:16inf1/4u}
\phi_{\infty,\frac{1}{4u}}(m,s,\chi_4)=\chi_4(-u)\phi_{0,\frac{1}{4(-u)}}(m,s,\chi_4)\\=\chi_4(-u)e\left(-\frac{mu}{4}\right)\frac{s(m)}{4^{2s}},\quad u=1,3,
\end{multline}
\begin{multline}\label{eq:16inf1/8}
\phi_{\infty,\frac{1}{8}}(m,s,\chi_4)=\phi_{0,1/2}(m,s,\chi_4)\\=\frac{2}{8^{2s}}\delta_{2}(m)t\left(\frac{m}{2}\right)-\frac{4}{(16)^{2s}}\delta_{4}(m)t\left(\frac{m}{4}\right),
\end{multline}
\begin{equation}\label{eq:16infinf}
\phi_{\infty,\infty}(m,s,\chi_4)=\phi_{0,0}(m,s,\chi_4)=\frac{4}{(16)^{2s}}\delta_{4}(m)t\left(\frac{m}{4}\right).
\end{equation}
\end{lem}

\begin{lem}
Let $N=4$. Then
\begin{equation}\label{eq:400}
\phi_{0,0}(m,s,\chi_4)=\frac{t(m)}{4^{2s}},\quad \phi_{0,\infty}(m,s,\chi_4)=\frac{s(m)}{2^{2s}},
\end{equation}
\begin{equation}\label{eq:4inf0}
\phi_{\infty,0}(m,s,\chi_4)=\chi_4(-1)\phi_{0,\infty}(m,s,\chi_4), 
\end{equation}
\begin{equation}\label{eq:4infinf}
 \phi_{\infty,\infty}(m,s,\chi_4)=\phi_{0,0}(m,s,\chi_4).
\end{equation}

\end{lem}

\begin{cor}\label{cor:zeros}
For $N=64$ we have
\begin{equation}\label{eq:zeroN64}
 \phi_{\infty,\frac{1}{32}}(m^2,s,\chi_4)=\phi_{0,\frac{1}{2}}(m^2,s,\chi_4)=0.
\end{equation}

For $N=16$ we have
\begin{equation}\label{eq:zeroN16}
 \phi_{\infty,\frac{1}{8}}(m^2,s,\chi_4)=\phi_{0,\frac{1}{2}}(m^2,s,\chi_4)=0.
\end{equation}
\end{cor}
\begin{proof}
The identity \eqref{eq:zeroN64} is a consequence of \eqref{eq:1-delta16}, \eqref{eq:64infphi1/32} and \eqref{eq:6401/2}.
Similarly, \eqref{eq:zeroN16} follows from \eqref{eq:16inf1/8} and \eqref{eq:1601/2}.
\end{proof}

\section{Contribution of the continuous spectrum}\label{sec:cont2}

Applying the Kuznetsov trace formula \eqref{eq:Kuznetsov} to the sums of Kloosterman sums in \eqref{eq:mndeven}, we find that for even $n$ the continuous spectrum \eqref{spectra:cont} can be written as a sum of 
\begin{multline}\label{c1even}
C^{1}_{\text{even}}:=-\frac{2i\zeta(2s)}{2^s\pi^{s-1/2}}\sum_{l=1}^{\infty}\frac{1}{l^s}
\sum_{\mathfrak{c} \text{ sing. }\Gamma_0(4)}\frac{1}{4\pi}\int_{-\infty}^{\infty}\frac{\psi_D(t)\sinh(\pi t)}{t\cosh{(\pi t)}}\\ \times l^{-2it} \overline{\phi_{\infty,\mathfrak{c}}(l^2,1/2+it,\chi_4)}
n^{2it}_{1} \phi_{\infty,\mathfrak{c}}(n^{2}_{1},1/2+it,\chi_4)dt
\end{multline}
and
\begin{multline}\label{c2even}
C^{2}_{\text{even}}:=\frac{2\zeta(2s)}{\pi^{s-1/2}}\sum_{l=1}^{\infty}\frac{1}{l^s}
\sum_{\mathfrak{c} \text{ sing. }\Gamma_0(4)}\frac{1}{4\pi}\int_{-\infty}^{\infty}\frac{\psi_D(t)\sinh(\pi t)}{t\cosh{(\pi t)}}\\ \times l^{-2it} \overline{\phi_{\infty,\mathfrak{c}}(l^2,1/2+it,\chi_4)}n_{1}^{2it}\phi_{0,\mathfrak{c}}(n_{1}^{2},1/2+it,\chi_4)dt.
\end{multline}
Similarly,  applying \eqref{eq:mndodd}  for odd $n$ it is required to investigate
\begin{multline}\label{c1odd}
C^{1}_{\text{odd}}:=\frac{8\zeta(2s)}{\pi^{s-1/2}}\sum_{l=1}^{\infty}\frac{1}{l^s}
\sum_{\mathfrak{c} \text{ sing. }\Gamma_0(64)}\frac{1}{4\pi}\int_{-\infty}^{\infty}\frac{\psi_D(t)\sinh(\pi t)}{t\cosh{(\pi t)}}\\ \times l^{-2it} \overline{\phi_{\infty,\mathfrak{c}}(l^2,1/2+it,\chi_4)}n^{2it}\phi_{0,\mathfrak{c}}(n^2,1/2+it,\chi_4)dt,
\end{multline}
\begin{multline}\label{c2odd}
C^{2}_{\text{odd}}:=\frac{4\zeta(2s)}{\pi^{s-1/2}}\sum_{l=1}^{\infty}\frac{1}{l^s}
\sum_{\mathfrak{c} \text{ sing. }\Gamma_0(16)}\frac{1}{4\pi}\int_{-\infty}^{\infty}\frac{\psi_D(t)\sinh(\pi t)}{t\cosh{(\pi t)}}\\ \times (1-2^{-s}) l^{-2it} \overline{\phi_{\infty,\mathfrak{c}}(l^2,1/2+it,\chi_4)}n^{2it}\phi_{0,\mathfrak{c}}(n^2,1/2+it,\chi_4)dt.
\end{multline}

In order to compute the sums over $l$ and $\mathfrak{c}$ in the expressions above we use the results of the previous section. Even and odd cases require separate treatment.
\subsection{Even case}
\begin{lem}\label{lem:odd111}
For $\Re{s}>1$ we have
\begin{multline}
C^{1}_{\text{even}}=\frac{L(\chi_4,s)}{4\pi^{s-1/2}2^{s}(1-2^{-2s})}
\frac{1}{2\pi i}\int_{-\infty}^{\infty}\frac{\psi_D(t)\sinh(\pi t)}{t\cosh{(\pi t)}}\\ \times
\frac{\zeta(s+2it)\zeta(s-2it)}{L(\chi_4,1+2it)L(\chi_4,1-2it)}\Biggl( (1-2^{2it-s})n_{1}^{2it}\sigma_{-2it}(\chi_4;n_{1}^{2})\\+ 
(1-2^{-2it-s})n_{1}^{-2it} \sigma_{2it}(\chi_4;n_{1}^{2})
\Biggr)dt.
\end{multline}

\end{lem}

\begin{proof}
There are two nonequivalent singular cusps for $\Gamma_0(4)$: $0$ and $\infty$. 

Consider first $\mathfrak{c}=\infty$.
Applying \eqref{eq:4infinf}, \eqref{eq:tm}, \eqref{eq:zs} with $z:=s-2it$ and $s:=-2it$, we compute the sum over $l$ in \eqref{c1even} 
\begin{multline}
\sum_{l=1}^{\infty}\frac{1}{l^s}
l^{-2it} \overline{\phi_{\infty,\mathfrak{\infty}}(l^2,1/2+it,\chi_4)}n_{1}^{2it}\phi_{\infty,\mathfrak{\infty}}(n_{1}^{2},1/2+it,\chi_4)\\=
\frac{|\tau(\chi_4)|^2}{16}\frac{1-2^{-2it-s}}{1-2^{-2s}}n_{1}^{-2it}\sigma_{2it}(\chi_4;n_{1}^{2})\frac{L(\chi_4,s)\zeta(s+2it)\zeta(s-2it)}{\zeta(2s)L(\chi_4,1+2it)L(\chi_4,1-2it)}.
\end{multline}

Now let us consider the case $\mathfrak{c}=0$. 
Using \eqref{eq:4inf0}, \eqref{eq:sm}, \eqref{eq:zs} with $z:=s+2it$ and $s:=2it$, we infer
\begin{multline}
\sum_{l=1}^{\infty}\frac{1}{l^s}
l^{-2it} \overline{\phi_{\infty,0}(l^2,1/2+it,\chi_4)}n_{1}^{2it}\phi_{\infty,0}(n_{1}^{2},1/2+it,\chi_4)\\=
\frac{1}{4}\frac{1-2^{2it-s}}{1-2^{-2s}}n_{1}^{2it}\sigma_{-2it}(\chi_4;n_{1}^{2})\frac{L(\chi_4,s)\zeta(s+2it)\zeta(s-2it)}{\zeta(2s)L(\chi_4,1+2it)L(\chi_4,1-2it)}.
\end{multline}

The assertion follows by summing the last two expressions and using the fact that $|\tau(\chi_4)|^2=4$.
\end{proof}

\begin{lem}\label{lem:even111}
For $\Re{s}>1$ the following identity holds
\begin{multline}
C^{2}_{\text{even}}=\frac{L(\chi_4,s)}{4\pi^{s-1/2}(1-2^{-2s})}
\frac{1}{2\pi i}\int_{-\infty}^{\infty}\frac{\psi_D(t)\sinh(\pi t)}{t\cosh{(\pi t)}}\\ \times
\frac{\zeta(s+2it)\zeta(s-2it)}{L(\chi_4,1+2it)L(\chi_4,1-2it)}
\Biggl( (1-2^{-2it-s})(2n_{1})^{2it}\sigma_{-2it}(\chi_4;n_{1}^{2})\\+ 
(1-2^{2it-s})(2n_{1})^{-2it}\sigma_{2it}(\chi_4;n_{1}^{2})
\Biggr)dt.
\end{multline}

\end{lem}

\begin{proof}
Similarly to the previous lemma, the sum over $\mathfrak{c}$ in \eqref{c2even} contains only two summands: $0$ and $\infty$. 

Let us start with $\mathfrak{c}=\infty$.
Applying \eqref{eq:4infinf}, \eqref{eq:400}, \eqref{eq:tm}, \eqref{eq:sm}, \eqref{eq:zs} with $z:=s-2it$ and $s:=-2it$, we obtain
\begin{multline}
\sum_{l=1}^{\infty}\frac{1}{l^s}
l^{-2it} \overline{\phi_{\infty,\mathfrak{\infty}}(l^2,1/2+it,\chi_4)}n_{1}^{2it}\phi_{0,\mathfrak{\infty}}(n_{1}^{2},1/2+it,\chi_4)\\=
\frac{\overline{\tau(\chi_4)}}{2^{3-2it}}\frac{1-2^{-2it-s}}{1-2^{-2s}}n_{1}^{2it}\sigma_{-2it}(\chi_4;n_{1}^{2})\frac{L(\chi_4,s)\zeta(s+2it)\zeta(s-2it)}{\zeta(2s)L(\chi_4,1+2it)L(\chi_4,1-2it)}.
\end{multline}

Next let us consider $\mathfrak{c}=0$.  Applying \eqref{eq:4inf0}, \eqref{eq:400},  \eqref{eq:tm},  \eqref{eq:sm}, \eqref{eq:zs} with $z:=s+2it$ and $s:=2it$, we find that
\begin{multline}
\sum_{l=1}^{\infty}\frac{1}{l^s}
l^{-2it} \overline{\phi_{\infty,0}(l^2,1/2+it,\chi_4)}n_{1}^{2it}\phi_{0,0}(n_{1}^{2},1/2+it,\chi_4)\\=
\frac{\tau(\chi_4)\chi_{4}(-1)}{2^{3+2it}}\frac{1-2^{2it-s}}{1-2^{-2s}}n_{1}^{-2it}\sigma_{2it}(\chi_4;n_{1}^{2})\frac{L(\chi_4,s)\zeta(s+2it)\zeta(s-2it)}{\zeta(2s)L(\chi_4,1+2it)L(\chi_4,1-2it)}.
\end{multline}

The assertion follows by summing the last two expressions and by noting that $\tau(\chi_4)=2i$ and $\chi_4(-1)=-1$.

\end{proof}

Using \eqref{eq:sigmachi4}, as a direct consequence of Lemmas \ref{lem:odd111} and \ref{lem:even111}, we obtain the following corollary.
\begin{cor}
For $Re{s}>1$ we have
\begin{multline}\label{eq:c1even+c2even}
C^{1}_{\text{even}}+C^{2}_{\text{even}}=\frac{L(\chi_4,s)}{4\pi^{s-1/2}}
\frac{1}{2\pi i}\int_{-\infty}^{\infty}\frac{\psi_D(t)\sinh(\pi t)}{t\cosh{(\pi t)}}\\ \times
\frac{\zeta(s+2it)\zeta(s-2it)}{L(\chi_4,1+2it)L(\chi_4,1-2it)} \Biggl( n^{2it}\sigma_{-2it}(\chi_4;n^2)+ 
n^{-2it}\sigma_{2it}(\chi_4;n^2)
\Biggr)dt.
\end{multline}
\end{cor}

Finally, we extend this result to the critical strip using the following lemma.
\begin{lem}\label{lem:Fs}
Suppose that the function $F(s)$ is defined for $\Re{s}>1$ by
\begin{equation}
F(s)=\frac{1}{2\pi i}\int_{(0)}f(s,z)dz,
\end{equation}
where $f(s,z)$ has two simple poles at  the points $z_{1}=1-s$ and $z_2=s-1$. Then for $\Re{s}<1$ we have
\begin{equation}
F(s)=\frac{1}{2\pi i}\int_{(0)}f(s,z)dz+\text{Res}_{z_1}f(s,z)-\text{Res}_{z_2}f(s,z).
\end{equation}
\end{lem}
\begin{proof}
Assume that $1<\Re{s}<1+\epsilon/2$ and $\Im{s}>0$. Consider the the new contour of integration as follows
\begin{equation*}
\gamma_1=(-i\infty,-i\Im{s}-i\epsilon)\cup C_{\epsilon}^{-}\cup (-i\Im{s}+i\epsilon,i\Im{s}-i\epsilon)\cup C_{\epsilon}^{+}\cup (i\Im{s}+i\epsilon,i\infty),
\end{equation*}
where $C_{\epsilon}^{-}$ is a semicircle in the left half-plane of radius $\epsilon$ and $C_{\epsilon}^{+}$ is a semicircle in the right half-plane of radius $\epsilon$.
While changing the contour of integration to $\gamma_1$ we cross poles at the points $z_1$, $z_2$. Therefore,
\begin{equation}
F(s)=\frac{1}{2\pi i}\int_{(\gamma_1)}f(s,z)dz+\text{Res}_{z_1}f(s,z)-\text{Res}_{z_2}f(s,z).
\end{equation}
Now if $\Re{s}<1$, we can change the contour back to $\Re{z}=0$. This completes the proof.
\end{proof}

\begin{lem}\label{lem:soh}
For $0<\Re{s}<1$, $s\neq 1/2$ we have
\begin{multline}\label{eq:c1c2even}
C^{1}_{\text{even}}+C^{2}_{\text{even}}=M^{C}(n,s)+\frac{L(\chi_4,s)}{4\pi^{s-1/2}}
\frac{1}{2\pi i}\int_{-\infty}^{\infty}\frac{\psi_D(t)\sinh(\pi t)}{t\cosh{(\pi t)}}\\ \times
\frac{\zeta(s+2it)\zeta(s-2it)}{L(\chi_4,1+2it)L(\chi_4,1-2it)}\Biggl( n^{2it}\sigma_{-2it}(\chi_4;n^2)+ 
n^{-2it}\sigma_{2it}(\chi_4;n^2)
\Biggr)dt,
\end{multline}
where $M^{C}(n,s)$ is defined by \eqref{eq:mcns}.
\end{lem}

\begin{proof}
Making the change of variables $z:=2it$ in \eqref{eq:c1even+c2even}, we have
\begin{multline}
C^{1}_{\text{even}}+C^{2}_{\text{even}}=\frac{L(\chi_4,s)}{4\pi^{s-1/2}}
\frac{1}{2\pi i}\int_{(0)}\frac{\psi_D\left(\frac{z}{2i}\right)\sinh\left(\pi \frac{z}{2i}\right)}{z\cosh\left(\pi \frac{z}{2i}\right)}\\ \times
\frac{\zeta(s+z)\zeta(s-z)}{L(\chi_4,1+z)L(\chi_4,1-z)} \Biggl( n^{z}\sigma_{-z}(\chi_4;n^2)+ 
n^{-z}\sigma_{z}(\chi_4;n^2)
\Biggr)dt.
\end{multline}
Then  \eqref{eq:c1c2even} follows from Lemma \ref{lem:Fs} and  \eqref{eq:psiD1-s}.

\end{proof}

\subsection{Odd case}
\begin{lem}
For $\Re{s}>1$ we have
\begin{multline}
C^{1}_{\text{odd}}=\frac{L(\chi_4,s)}{4\pi^{s-1/2}(1-2^{-2s})}
\frac{1}{2\pi i}\int_{-\infty}^{\infty}\frac{\psi_D(t)\sinh(\pi t)}{t\cosh{(\pi t)}}\\ \times
\frac{\zeta(s+2it)\zeta(s-2it)}{L(\chi_4,1+2it)L(\chi_4,1-2it)}n^{2it}\sigma_{-2it}(\chi_4;n^2)\\ \times\Biggl( 
(1-2^{2it-s})(1-2^{-2it-s})+\frac{2^{-2it}+2^{2it}-2^{1-s}}{2^{2s}}
\Biggr)dt.
\end{multline}

\end{lem}
\begin{proof}

Let us decompose the sum over $l$ in \eqref{c1odd} as follows
$$\sum_{l=1}^{\infty}=\sum_{l\equiv 1\pmod{2}}+\sum_{l\equiv 0\pmod{2}},$$
so that $C^{1}_{\text{odd}}=C^{1,1}_{\text{odd}}+C^{1,2}_{\text{odd}}$.

Consider first $C^{1,1}_{\text{odd}}$.
 By Lemma \ref{lem:640} we obtain for odd $n$ that
\begin{equation}\label{eq:phi0c}
 \phi_{0,\mathfrak{c}}(n^2,1/2+it,\chi_4)=0, \quad \mathfrak{c}=0, \frac{1}{2}, \frac{1}{4}, \frac{1}{12}. 
 \end{equation}
 
 If $l$ is odd, it follows from Lemma \ref{lem:64inf} that
\begin{equation}
 \phi_{\infty,\mathfrak{c}}(l^2,1/2+it,\chi_4)=0, \quad \mathfrak{c}=\infty, \frac{1}{32}, \frac{1}{16}, \frac{1}{48}. 
 \end{equation}

It is left to consider the case $\mathfrak{c}=\frac{1}{8u}$, $u=1,3,5,7$. Note that for odd  $n$ we have $e(n^2u/8)=e(u/8)$. Consequently, applying \eqref{eq:sm}, \eqref{eq:64inf1/8u} and \eqref{eq:6401/8u}  we obtain 
\begin{multline}\label{eq:sumoverlodd}
\sum_{(l,2)=1}\frac{1}{l^{s+2it}}
 \overline{\phi_{\infty,1/(8u)}(l^2,1/2+it,\chi_4)}n^{2it}\phi_{0,1/(8u)}(n^2,1/2+it,\chi_4)\\=
\frac{n^{2it}\sigma_{-2it}(\chi_4;n^2)\chi_4(-u)e(u/4)}{64L(\chi_4,1+2it)L(\chi_4,1-2it)}\sum_{(l,2)=1}\frac{\sigma_{2it}(\chi_4;l^2)}{l^{s+2it}}.
\end{multline}

Using \eqref{eq:zs} with $z:=s+2it$, $s:=2it$ we show that
\begin{multline}\label{eq:sumoverlodd1}
\sum_{(l,2)=1}\frac{\sigma_{2it}(\chi_4;l^2)}{l^{s+2it}}=\left( 1-2^{-s-2it}\right)\sum_{l=1}^{\infty}\frac{\sigma_{2it}(\chi_4;l^2)}{l^{s+2it}}\\
=\left( 1-2^{-s-2it}\right)\frac{1-2^{2it-s}}{1-2^{-2s}}\frac{L(\chi_4,s)\zeta(s+2it)\zeta(s-2it)}{\zeta(2s)}.
\end{multline}

Substituting \eqref{eq:sumoverlodd1} into  \eqref{eq:sumoverlodd}, summing over $u=1,3,5,7$ and using  the identity
\begin{equation}
\sum_{u=1,3,5,7}\chi_4(-u)e(u/4)=-4i,
\end{equation}
we infer
\begin{multline}\label{eq:c11}
C^{1,1}_{\text{odd}}=\frac{L(\chi_4,s)}{4\pi^{s-1/2}(1-2^{-2s})}
\frac{1}{2\pi i}\int_{-\infty}^{\infty}\frac{\psi_D(t)\sinh(\pi t)}{t\cosh{(\pi t)}}\sigma_{-2it}(\chi_4;n^2)\\ \times n^{2it}
\frac{\zeta(s+2it)\zeta(s-2it)}{L(\chi_4,1+2it)L(\chi_4,1-2it)}
(1-2^{2it-s})(1-2^{-2it-s})dt.
\end{multline}

Now let us evaluate $C^{1,2}_{\text{odd}}$. Using \eqref{eq:phi0c}  we conclude that it is left to consider the following cases: $\mathfrak{c}=\frac{1}{8u}$, $u=1,3,5,7$ and  $\mathfrak{c}=\frac{1}{16}, \frac{1}{48}$,  $\mathfrak{c}=\frac{1}{32}$,  $\mathfrak{c}=\infty$.

Assume that $\mathfrak{c}=\frac{1}{8u}$. Using the fact that $e(n^2u/8)=e(u/8)$ for odd  $n$, and applying \eqref{eq:sm}, \eqref{eq:64inf1/8u} and \eqref{eq:6401/8u}  we obtain 
\begin{multline}
\sum_{l\equiv 0\pmod{2}}\frac{1}{l^{s+2it}}
 \overline{\phi_{\infty,1/(8u)}(l^2,1/2+it,\chi_4)}n^{2it}\phi_{0,1/(8u)}(n^2,1/2+it,\chi_4)\\=
\frac{n^{2it}\sigma_{-2it}(\chi_4;n^2)\chi_4(-u)e(u/8)}{64L(\chi_4,1+2it)L(\chi_4,1-2it)}\sum_{l\equiv 0\pmod{2}}\frac{e(l^2u/8)\sigma_{2it}(\chi_4;l^2)}{l^{s+2it}}.
\end{multline}

Note that $e(m^2u/2)=1$ if $m$ is even and $e(m^2u/2)=-1$ if $m$ is odd. Therefore, 
\begin{multline}
\sum_{l\equiv 0\pmod{2}}\frac{e(l^2u/8)\sigma_{2it}(\chi_4;l^2)}{l^{s+2it}}=\frac{1}{2^{s+2it}}\sum_{l=1}^{\infty}\frac{e(l^2u/2)\sigma_{2it}(\chi_4;l^2)}{l^{s+2it}}\\
=\frac{1}{2^{s+2it}}\sum_{l\equiv 0 \pmod{2}}\frac{\sigma_{2it}(\chi_4;l^2)}{l^{s+2it}}-\frac{1}{2^{s+2it}}\sum_{l\equiv 1 \pmod{2}}\frac{\sigma_{2it}(\chi_4;l^2)}{l^{s+2it}}.
\end{multline}
Since 
\begin{equation}
\sum_{u=1,3,5,7}\chi_4(-u)e(u/8)=0,
\end{equation}
the contribution of the part  with $\mathfrak{c}=\frac{1}{8u}$, $u=1,3,5,7$ to  $C^{1,2}_{\text{odd}}$ is zero.

Now assume that $\mathfrak{c}=\frac{1}{16u}$, $u=1,3$. Applying \eqref{eq:sm}, \eqref{eq:64inf1/16u} and \eqref{eq:6401/16u}, we have
\begin{multline}\label{eq:l022}
\sum_{(l,2)=1}\frac{n^{2it}}{l^{s+2it}}
 \overline{\phi_{\infty,1/(16u)}(l^2,1/2+it,\chi_4)}\phi_{0,1/(16u)}(n^2,1/2+it,\chi_4)\\=
\frac{n^{2it}\sigma_{-2it}(\chi_4;n^2)\chi_4(-u)e(u/4)}{2^{5-2it}L(\chi_4,1+2it)L(\chi_4,1-2it)}\sum_{l\equiv 0\pmod{2}}\frac{e(l^2u/16)\sigma_{2it}(\chi_4;l^2)}{l^{s+2it}}.
\end{multline}

Using \eqref{eq:sigmachi4} we infer
\begin{multline}
\sum_{l\equiv 0\pmod{2}}\frac{e(l^2u/16)\sigma_{2it}(\chi_4;l^2)}{l^{s+2it}}=\frac{1}{2^{s+2it}}\sum_{l=1}^{\infty}\frac{e(l^2u/4)\sigma_{2it}(\chi_4;l^2)}{l^{s+2it}}\\
=\frac{1}{4^{s+2it}}\sum_{l=1}^{\infty}\frac{\sigma_{2it}(\chi_4;l^2)}{l^{s+2it}}+\frac{e(u/4)}{2^{s+2it}}\sum_{(l,2)=1}\frac{\sigma_{2it}(\chi_4;l^2)}{l^{s+2it}}.
\end{multline}
Rewriting the sum over $l$ as 
$$\sum_{(l,2)=1}=\sum_{l=1}^{\infty}-\sum_{l \equiv 0\pmod{2}}$$
and applying \eqref{eq:zs} with $z:=s+2it$, $s:=2it$, we show that
\begin{multline}\label{eq:l02}
\sum_{l\equiv 0\pmod{2}}\frac{e(l^2u/16)\sigma_{2it}(\chi_4;l^2)}{l^{s+2it}}
=\frac{1-2^{2it-s}}{2^{s+2it}(1-2^{-2s})}
\\ \times \frac{L(\chi_4,s)\zeta(s+2it)\zeta(s-2it)}{\zeta(2s)}
\left(2^{-s-2it}+e(u/4)\left( 1-2^{-s-2it}\right) \right).
\end{multline}

Substituting \eqref{eq:l02} into \eqref{eq:l022} and summing over $u=1,2$ we conclude that the contribution of the part  with $\mathfrak{c}=\frac{1}{16u}$, $u=1,2$ to  $C^{1,2}_{\text{odd}}$ is equal to
\begin{multline}\label{mult:c10dd1}
\frac{n^{2it}\sigma_{-2it}(\chi_4;n^2)}{16i}\frac{1-2^{2it-s}}{2^{2s+2it}(1-2^{-2s})}
\\ \times \frac{L(\chi_4,s)\zeta(s+2it)\zeta(s-2it)}{\zeta(2s)L(\chi_4,1+2it)L(\chi_4,1-2it)}.
\end{multline}

Next, assume that $\mathfrak{c}=\frac{1}{32}$. By Corollary \ref{cor:zeros} we have
$$\phi_{\infty,1/(32)}(l^2,1/2+it,\chi_4)=0.$$

Finally, assume that $\mathfrak{c}=\infty$. As a consequence of  \eqref{eq:64infphi1/64},  \eqref{eq:0infphi1/64}, \eqref{eq:sm}, \eqref{eq:tm}, \eqref{eq:zs}, we obtain
\begin{multline}\label{mult:c10dd2}
\sum_{l \equiv 0\pmod{2}}\frac{n^{2it}}{l^{s+2it}}
 \overline{\phi_{\infty,\infty}(l^2,1/2+it,\chi_4)}\phi_{0,\infty}(n^2,1/2+it,\chi_4)=\\
\frac{n^{2it}\sigma_{-2it}(\chi_4;n^2)}{16i}\frac{1-2^{-2it-s}}{2^{2s-2it}(1-2^{-2s})}
 \frac{L(\chi_4,s)\zeta(s+2it)\zeta(s-2it)}{\zeta(2s)L(\chi_4,1+2it)L(\chi_4,1-2it)}.
\end{multline}

Combining  \eqref{c1odd}, \eqref{mult:c10dd1}, \eqref{mult:c10dd2} we find that
\begin{multline}\label{eq:c12}
C^{1,2}_{\text{odd}}=\frac{L(\chi_4,s)}{4\pi^{s-1/2}(1-2^{-2s})}
\frac{1}{2\pi i}\int_{-\infty}^{\infty}\frac{\psi_D(t)\sinh(\pi t)}{t\cosh{(\pi t)}}\sigma_{-2it}(\chi_4;n^2)\\ \times n^{2it}
\frac{\zeta(s+2it)\zeta(s-2it)}{L(\chi_4,1+2it)L(\chi_4,1-2it)}\Biggl( 
\frac{2^{-2it}+2^{2it}-2^{1-s}}{2^{2s}}
\Biggr)dt.
\end{multline}

The statement follows  by summing \eqref{eq:c11} and  \eqref{eq:c12}.

\end{proof}
\begin{lem}
For $\Re{s}>1$ the following identity holds
\begin{multline}
C^{2}_{\text{odd}}=\frac{L(\chi_4,s)(1-2^{-s})}{4\pi^{s-1/2}(1-2^{-2s})}
\frac{1}{2\pi i}\int_{-\infty}^{\infty}\frac{\psi_D(t)\sinh(\pi t)}{t\cosh{(\pi t)}}\sigma_{-2it}(\chi_4;n^2)\\ \times n^{2it}
\frac{\zeta(s+2it)\zeta(s-2it)}{L(\chi_4,1+2it)L(\chi_4,1-2it)}
\frac{2^{-2it}+2^{2it}-2^{1-s}}{2^{s}}dt.
\end{multline}

\end{lem}
\begin{proof}

The list of singular cusps for $\Gamma_0(16)$ is given in Lemma \ref{lem:singularcusps}. 

Consider $\mathfrak{c}=0$ and $\mathfrak{c}=1/2$ . As a consequence of \eqref{eq:1600} and Corollary \ref{cor:zeros} we obtain that for odd $n$
\begin{equation}
 \phi_{0,0}(n^2,1/2+it,\chi_4)=\phi_{0,1/2}(n^2,1/2+it,\chi_4)=0.
 \end{equation}
 
Consider $\mathfrak{c}=\frac{1}{8}$. By Corollary \ref{cor:zeros} we have
\begin{equation}
\phi_{\infty,1/8}(l^2,1/2+it,\chi_4)=0.
\end{equation}

Consider $\mathfrak{c}=\frac{1}{4u}$, $u=1,3$. Using \eqref{eq:16inf1/4u}, \eqref{eq:sm}, we compute the sum over $l$ in \eqref{c2odd}
\begin{multline}\label{eq:oddcasechi41}
\sum_{l=1}^{\infty}\frac{\overline{\phi_{\infty,1/(4u)}(l^2,1/2+it,\chi_4)}}{l^{s+2it}} \\=
\frac{\chi_4(-u)}{4^{1-2it}L(\chi_4,1-2it)}\sum_{l=1}^{\infty}\frac{\sigma_{2it}(\chi_4;l^2)e(l^2u/4)}{l^{s+2it}}.
\end{multline}

The last sum can be split into two parts
\begin{multline}
\sum_{l=1}^{\infty}\frac{\sigma_{2it}(\chi_4;l^2)e(l^2u/4)}{l^{s+2it}}=\frac{1}{2^{s+2it}}\sum_{l=1}^{\infty}\frac{\sigma_{2it}(\chi_4;4l^2)}{l^{s+2it}}\\+e(u/4)\sum_{l\equiv 1\pmod{2}}\frac{\sigma_{2it}(\chi_4;l^2)}{l^{s+2it}}.
\end{multline}

Furthermore,
\begin{equation}
\sum_{l\equiv 1\pmod{2}}\frac{\sigma_{2it}(\chi_4;l^2)}{l^{s+2it}}=\sum_{l=1}^{\infty}\frac{\sigma_{2it}(\chi_4;l^2)}{l^{s+2it}}-\frac{1}{2^{s+2it}}\sum_{l=1}^{\infty}\frac{\sigma_{2it}(\chi_4;4l^2)}{l^{s+2it}}.
\end{equation}
Note that  $\sigma_{2it}(\chi_4;4l^2)=\sigma_{2it}(\chi_4;l^2)$, and therefore,
\begin{multline}\label{eq:oddcasechi4}
\sum_{l=1}^{\infty}\frac{\sigma_{2it}(\chi_4;l^2)e(l^2u/4)}{l^{s+2it}}\\=\left( \frac{1}{2^{s+2it}}+ e(u/4)(1-2^{-s-2it})\right)
\sum_{l=1}^{\infty}\frac{\sigma_{2it}(\chi_4;4l^2)}{l^{s+2it}}.
\end{multline}
Substituting \eqref{eq:oddcasechi4} into \eqref{eq:oddcasechi41} and applying \eqref{eq:zs} with $z:=s+2it$ and $s:=2it$ yields
\begin{multline}
\sum_{l=1}^{\infty}\frac{\overline{\phi_{\infty,1/(4u)}(l^2,1/2+it,\chi_4)}}{l^{s+2it}} =
\left( \frac{1}{2^{s+2it}}+ e(u/4)(1-2^{-s-2it})\right)\\ \times \frac{\chi_4(-u)}{4^{1-2it}L(\chi_4,1-2it)}\frac{1-2^{2it-s}}{1-2^{-2s}}\frac{L(\chi_4,s)\zeta(s+2it)\zeta(s-2it)}{\zeta(2s)}.
\end{multline}
Using \eqref{eq:1601/4u}, \eqref{eq:sm}, and the fact that $n$ is odd, we infer
\begin{multline}\label{eq:suml04u}
\sum_{l=1}^{\infty}\frac{\overline{\phi_{\infty,1/(4u)}(l^2,1/2+it,\chi_4)}}{l^{s+2it}} n^{2it}\phi_{0,1/(4u)}(n^2,1/2+it,\chi_4)\\=
\left( \frac{1}{2^{s+2it}}+ e(u/4)(1-2^{-s-2it})\right) n^{2it}\sigma_{-2it}(\chi_4;n^2)\\ \times \frac{\chi_4(-u)e(u/4)}{16}\frac{1-2^{2it-s}}{1-2^{-2s}}\frac{L(\chi_4,s)\zeta(s+2it)\zeta(s-2it)}{\zeta(2s)L(\chi_4,1-2it)L(\chi_4,1+2it)}.
\end{multline}

Consider $\mathfrak{c}=\infty$. In order to evaluate the sum over $l$ in this case we apply  \eqref{eq:16infinf}, \eqref{eq:tm}, and \eqref{eq:zs} with $z:=s-2it$, $s:=-2it$. Furthermore, using \eqref{eq:160inf},  \eqref{eq:sm} and the fact that $\tau(\chi_4)=2i$, we obtain
\begin{multline}\label{eq:suml0inf}
\sum_{l=1}^{\infty}\frac{1}{l^s} l^{-2it} \overline{\phi_{\infty,\infty}(l^2,1/2+it,\chi_4)}n^{2it}\phi_{0,\infty}(n^2,1/2+it,\chi_4)\\
=\frac{-i}{2^{3+s-2it}}\frac{1-2^{-2it-s}}{1-2^{-2s}}n^{2it}\sigma_{-2it}(\chi_4;n^2)\frac{L(\chi_4,s)\zeta(s+2it)\zeta(s-2it)}{\zeta(2s)L(\chi_4,1+2it)L(\chi_4,1-2it)}.
\end{multline}

Summing \eqref{eq:suml0inf}, \eqref{eq:suml04u} for $u=1,3$ and noting that
\begin{equation}
\sum_{u=1,3}\chi_4(-u)=0, \quad \sum_{u=1,3}\chi_4(-u)e(u/4)=-2i,
\end{equation}
we prove the lemma.

\end{proof}

Combining the previous two lemmas, we prove the following result.

\begin{cor}
For $\Re{s}>1$ we have
\begin{multline}
C^{1}_{\text{odd}}+C^{2}_{\text{odd}}=\frac{L(\chi_4,s)}{4\pi^{s-1/2}}
\frac{1}{2\pi i}\int_{-\infty}^{\infty}\frac{\psi_D(t)\sinh(\pi t)}{t\cosh{(\pi t)}}\\ \times
\frac{\zeta(s+2it)\zeta(s-2it)}{L(\chi_4,1+2it)L(\chi_4,1-2it)}
n^{2it}\sigma_{-2it}(\chi_4;n^2)dt.
\end{multline}

\end{cor}

\begin{lem}\label{lem:soh2}
For $0<\Re{s}<1$, $s\neq 1/2$ the following holds
\begin{multline}
C^{1}_{\text{odd}}+C^{2}_{\text{odd}}=\frac{1}{2}M^{C}(n,s)+\frac{L(\chi_4,s)}{4\pi^{s-1/2}}
\frac{1}{2\pi i}\int_{-\infty}^{\infty}\frac{\psi_D(t)\sinh(\pi t)}{t\cosh{(\pi t)}}\\ \times
\frac{\zeta(s+2it)\zeta(s-2it)}{L(\chi_4,1+2it)L(\chi_4,1-2it)}
n^{2it}\sigma_{-2it}(\chi_4;n^2)dt,
\end{multline}
where $M^{C}(n,s)$ is defined by \eqref{eq:mcns}.
\end{lem}
\begin{proof}
The proof is the same as the one of Lemma \ref{lem:soh}.
\end{proof}

\section{Contribution of the discrete and homomorphic spectra}\label{sect:dischol}

Applying the Kuznetsov trace formula \eqref{eq:Kuznetsov} to the sums of Kloosterman sums in \eqref{eq:mndeven}  for even $n$, we find that the discrete spectrum is a sum of 
two twisted moments of symmetric square $L$-functions associated to Maass cusp forms of level $4$ with nebentypus $\chi_4$, namely
\begin{equation}
D_{1}^{\text{even}}=-\frac{2^{1-s}\pi^{1/2-s}i}{1-2^{-2s}}\sum_{f\in H(4,\chi_4)}\frac{\psi_{D}(t_f)}{\cosh(\pi t_f)}\rho_{f_{\infty}}(n^{2}/4)\overline{L(s,\text{sym}^2 f_{\infty})}
\end{equation}
and
\begin{equation}
D_{2}^{\text{even}}=\frac{2\pi^{1/2-s}}{1-2^{-2s}}\sum_{f\in H(4,\chi_4)}\frac{\psi_{D}(t_f)}{\cosh(\pi t_f)}\rho_{f_{0}}(n^{2}/4)\overline{L(s,\text{sym}^2 f_{\infty})}.
\end{equation}
Similarly, the holomorphic spectrum is a sum of two twisted moments of symmetric square $L$-functions associated to holomorphic cusp forms of level $4$ with nebentypus $\chi_4$
given by 
\begin{multline}
H_{1}^{\text{even}}=-\frac{2^{1-s}\pi^{1/2-s} i}{1-2^{-2s}}\sum_{\substack{k>1\\k \text{ odd}}}\psi_{H}(k)\Gamma(k)
\sum_{f\in H_k(4,\chi_4)}\rho_{f_{\infty}}\left(\frac{n^{2}}{4}\right)\overline{L(s,\text{sym}^2 f_{\infty})}
\end{multline}
and 
\begin{multline}
H_{2}^{\text{even}}=\frac{2\pi^{1/2-s}}{1-2^{-2s}}\sum_{\substack{k>1\\k \text{ odd}}}\psi_{H}(k)\Gamma(k)
\sum_{f\in H_k(4,\chi_4)}\rho_{f_{0}}\left(\frac{n^{2}}{4}\right)\overline{L(s,\text{sym}^2 f_{\infty})}.
\end{multline}

Next let us consider the case of odd $n$. The main difference with the previous case is that now we obtain moments of $L$-functions associated to forms of levels $64$ and $16$.
Indeed, applying the Kuznetsov trace formula \eqref{eq:Kuznetsov} to the sums of Kloosterman sums in \eqref{eq:mndodd}, we infer that the discrete spectrum consists of
\begin{equation}
D_{1}^{\text{odd}}=\frac{8\pi^{1/2-s}}{1-2^{-2s}}\sum_{f\in H(64,\chi_4)}\frac{\psi_{D}(t_f)}{\cosh(\pi t_f)}\rho_{f_{0}}(n^{2})\overline{L(s,\text{sym}^2 f_{\infty})}
\end{equation}
and
\begin{equation}
D_{2}^{\text{odd}}=\frac{4\pi^{1/2-s}}{1+2^{-s}}\sum_{f\in H(16,\chi_4)}\frac{\psi_{D}(t_f)}{\cosh(\pi t_f)}\rho_{f_{0}}(n^{2})\overline{L(s,\text{sym}^2 f_{\infty})}.
\end{equation}

Similarly, the holomorphic spectrum is a sum of these two parts:
\begin{multline}
H_{1}^{\text{odd}}=\frac{8\pi^{1/2-s}}{1-2^{-2s}}\sum_{\substack{k>1\\k \text{ odd}}}\psi_{H}(k)\Gamma(k)
\sum_{f\in H_k(64,\chi_4)}\rho_{f_{0}}(n^{2})\overline{L(s,\text{sym}^2 f_{\infty})},
\end{multline}

\begin{multline}
H_{2}^{\text{odd}}=\frac{4\pi^{1/2-s}}{1+2^{-s}}\sum_{\substack{k>1\\k \text{ odd}}}\psi_{H}(k)\Gamma(k)
\sum_{f\in H_k(16,\chi_4)}\rho_{f_{0}}(n^{2})\overline{L(s,\text{sym}^2 f_{\infty})}.
\end{multline}


\section{Proof of main theorems}\label{sec:mainthm}

\subsection{Proof of Theorems \ref{thm:main1} and \ref{thm:main2}}

The diagonal main terms in Theorems  \ref{thm:main1} and \ref{thm:main2} are computed in Lemma \ref{lem:diagonal}.
In order to evaluate the non-diagonal part, we apply the Kuznetsov trace formula \eqref{eq:Kuznetsov} to the sums of Kloosterman sums in \eqref{eq:mndeven} and \eqref{eq:mndodd}.
It follows from Lemmas \ref{lem:soh} and \ref{lem:soh2} that the contribution of the continuous spectrum is equal to $M^{C}(n,s)+\mathfrak{C}(n,s)$ in case of even $n$ and to $1/2M^{C}(n,s)+1/2\mathfrak{C}(n,s)$ in case of odd $n$,
where 
\begin{equation}
\mathfrak{C}(n,s)=C_{\text{even}}^{1}+C_{\text{even}}^{2}-M^{C}(n,s).
\end{equation}
Finally, the contribution of the discrete and holomorphic spectra is given in Section \ref{sect:dischol}. For the sake of brevity, we express the final result in terms of sums of moments \eqref{moment} using the following identites:
\begin{equation}
D_{\text{even}}^{1}+H_{\text{even}}^{1}=-\frac{2^{1-s}\pi^{1/2-s}i}{1-2^{-2s}}\mathfrak{M}_{\infty}(n^2/4,4,s),
\end{equation}
\begin{equation}
D_{\text{even}}^{2}+H_{\text{even}}^{2}=\frac{2\pi^{1/2-s}}{1-2^{-2s}}\mathfrak{M}_{0}(n^2/4,4,s),
\end{equation}
\begin{equation}
D_{\text{odd}}^{1}+H_{\text{odd}}^{1}=\frac{8\pi^{1/2-s}}{1-2^{-2s}}\mathfrak{M}_{0}(n^2,64,s),
\end{equation}
\begin{equation}
D_{\text{odd}}^{2}+H_{\text{odd}}^{2}=\frac{4\pi^{1/2-s}}{1+2^{-s}}\mathfrak{M}_{0}(n^2,16,s).
\end{equation}

\subsection{Main terms at the central point}\label{section:holomorphic}

As the final step we show that the main terms in Theorems \ref{thm:main1} and \ref{thm:main2} are holomorphic at the central point.

\begin{lem}
For $n$ even the following identity holds
\begin{multline}
M^{C}(n,1/2)+M^{D}_{\text{even}}(n,1/2)=
\\ \frac{\sigma_{-1/2}(\chi_4;n^2)+n^{-1}\sigma_{1/2}(\chi_4;n^2)}{2L(\chi_4,3/2)} \int_{0}^{\infty}\omega(y)\\ 
\times
\left( \log|y^2-n^2/4|+\frac{\pi}{2}\text{sgn}(y-n/2)-2\frac{L'(\chi_4,3/2)}{L(\chi_4,3/2)}-\log(2\pi)+3\gamma
\right)dy\\
-\frac{\sigma'_{-1/2}(\chi_4;n^2)-n^{-1}\sigma'_{1/2}(\chi_4;n^2)+2\log(n)n^{-1}\sigma_{1/2}(\chi_4;n^2)}{2L(\chi_4,3/2)}\int_{0}^{\infty}\omega(y)dy.
\end{multline}
\end{lem}

\begin{proof}
As a consequence of the functional equations for the Riemann zeta function and for the Gamma function, we obtain
\begin{equation}\label{eq:zetagamma}
\zeta(2u)\Gamma(u)=(2\pi)^{2u}\frac{\Gamma(1-2u)}{\Gamma(1-u)}\zeta(1-2u).
\end{equation}
Combining   \eqref{eq:mevendef}, \eqref{eq:mcns} and \eqref{eq:zetagamma} we conclude that
\begin{multline}
M^{C}(n,1/2+u)+M^{D}_{\text{even}}(n,1/2+u)=\\
\zeta(1+2u)\frac{\sigma_{-1/2-u}(\chi_4;n^2)+n^{-1-2u}\sigma_{1/2+u}(\chi_4;n^2)}{L(\chi_4,3/2+u)}
\int_{0}^{\infty}\omega(y)dy\\
+\zeta(1-2u)(2\pi)^{u}\frac{\Gamma(1-2u)}{\Gamma(1-u)}\frac{\sigma_{-1/2+u}(\chi_4;n^2)+n^{-1+2u}\sigma_{1/2-u}(\chi_4;n^2)}{L(\chi_4,3/2-u)}\\ 
\times \sqrt{2}
\biggl(
\sin (\pi/4+\pi u/2)\int_{0}^{n/2}\omega(y)\left( \frac{n^2}{4}-y^2\right)^{-u}dy\\
+\cos(\pi/4+\pi u/2)\int_{n/2}^{\infty}\omega(y)\left(y^2- \frac{n^2}{4}\right)^{-u}dy\biggr).
\end{multline}
The expression above is holomorphic at $u=0$. Consequently, letting $u$ tend to zero and applying the L'H\^{o}spital rule, we prove the lemma.
\end{proof}


\begin{lem}
For $n$ odd the following identity holds
\begin{multline}
\frac{1}{2}M^{C}(n,1/2)+M^{D}_{\text{odd}}(n,1/2)=\frac{\sigma_{-1/2}(\chi_4;n^2)}{2L(\chi_4,3/2)} \int_{0}^{\infty}\omega(y) \times \\
\left( \log|y^2-n^2/4|+\frac{\pi}{2}\text{sgn}(y-n/2)-2\frac{L'(\chi_4,3/2)}{L(\chi_4,3/2)}-\log(2\pi)+3\gamma
\right)dy\\
-\frac{\sigma'_{-1/2}(\chi_4;n^2)}{2L(\chi_4,3/2)}\int_{0}^{\infty}\omega(y)dy.
\end{multline}
\end{lem}

\begin{proof}
Using \eqref{eq:mcns},  \eqref{eq:mdodddef}, \eqref{eq:sigmatwist} and \eqref{eq:zetagamma} we obtain
\begin{multline}
\frac{1}{2}M^{C}(n,1/2+u)+M^{D}_{\text{odd}}(n,1/2+u)=\\
\zeta(1+2u)\frac{\sigma_{-1/2-u}(\chi_4;n^2)}{L(\chi_4,3/2+u)} \int_{0}^{\infty}\omega(y)dy\\
+\zeta(1-2u)(2\pi)^{u}\frac{\Gamma(1-2u)}{\Gamma(1-u)}\frac{\sigma_{-1/2+u}(\chi_4;n^2)}{L(\chi_4,3/2-u)}\\ 
\times \sqrt{2}
\biggl(
\sin (\pi/4+\pi u/2)\int_{0}^{n/2}\omega(y)\left( \frac{n^2}{4}-y^2\right)^{-u}dy\\
+\cos(\pi/4+\pi u/2)\int_{n/2}^{\infty}\omega(y)\left(y^2- \frac{n^2}{4}\right)^{-u}dy\biggr).
\end{multline}
The expression above is holomorphic at $u=0$. Therefore, the assertion follows by letting $u$ tend to zero and applying the L'H\^{o}spital rule.
\end{proof}

\section*{Acknowledgement}
The reported study was funded by RFBR, project number 19-31-60029.

\nocite{}

\end{document}